%% file: main.tex
\documentclass[11pt]{article}
\usepackage{latexsym,enumerate}
\usepackage{graphicx,color,amsmath,amsthm,amsopn,amstext,amscd,amsfonts,amssymb}
\usepackage{vuk,a4wide,url,xspace,bbm}
\usepackage[american]{babel}
\usepackage{setspace}
\author{Vuk Mili{{\v s}}i{{\'c}}\footnotemark[1]\,\,\footnotemark[2]}
\title{Very weak estimates for a rough Poisson-Dirichlet problem
with natural vertical boundary conditions}
\begin{document}
\renewcommand{\thefootnote}{\fnsymbol{footnote}}
\footnotetext[1]{Wolfgang Pauli Institute (WPI), UMI CNRS 2841,Vienna, AUSTRIA}
\footnotetext[2]{Laboratoire Jean Kuntzman (LJK), UMR CNRS 5523, Grenoble, FRANCE}
\renewcommand{\thefootnote}{\arabic{footnote}}

\maketitle

\begin{abstract}
This work is a  continuation of \cite{BrBoMi}; it
deals with rough boundaries in the simplified context
of a Poisson equation. We impose Dirichlet boundary
conditions on the periodic microscopic perturbation of a flat
edge on one side and natural homogeneous Neumann boundary
conditions are applied on the inlet/outlet of the domain.
To prevent oscillations on the Neumann-like boundaries, we
introduce a microscopic vertical corrector defined
in a rough quarter-plane. In \cite{BrBoMi} we studied {\em a priori}
estimates in this setting;  here we fully develop very weak estimates
{ \em {\`a} la } Ne{\v c}as \cite{Ne.Book.67}  in the weighted Sobolev spaces
on an unbounded domain. We obtain optimal estimates
which improve those  derived in \cite{BrBoMi}. We validate
these results numerically, proving first order results
for boundary layer approximation including the 
vertical correctors and a little less for the averaged 
wall-law introduced in the literature \cite{JaMiJDE.01,NeNeMi.06}.
\end{abstract}

\bigskip


\bigskip


\bigskip

\section{Introduction}

Cardio-vascular pathologies of the arterial wall represent 
a challenging area of investigation since they are one of the 
major cause of death in occidental countries. 
In this context, we are strongly interested in the accurate
description of blood-flow characteristics in stented arteries. Specifically,
we aim to understand the influence of a metallic wired stent (a medical device that 
cures some of these pathologies) on the circulatory 
system: our goal is to give a detailed description 
of the flow upward, inward and backward the region of stent's location. 
Actually the stent could be seen as a local perturbation
of a smooth boundary of the flow field.
The change,  from perturbed to smooth, strongly 
contradicts the hypothesis of periodicity 
faced by the author in \cite{BrMiQam,BrMiCras}.

Although this problem was tackled in 
\cite{JaMiSIAM.00,JaMiJDE.01}, our formalism
follows ideas presented in \cite{SaPeZa.85}
for interior homogenization problems, and it 
should be easy to extend it to other linear elliptic operators.
A first step in this direction was made in \cite{BrBoMi}
for a simplified Poisson problem: we set up a formal
approach to handle natural boundary conditions
at the inlet and the outlet of a straight rough domain;
then we proved rigorously, {\em via} specific {\em a priori} estimates,
that the boundary layer approximation - built by adding some 
vertical correctors - converges 
to the exact solution of the rough problem. These estimates
validated our approach.

In \cite{JaMiSIAM.00,JaMiJDE.01}, the authors
introduced, {\em via} very weak solutions \cite{Ne.Book.67}, $L^2$ estimates
of the error between various approximations and the
exact rough solution. These estimates were established in a 
 piecewise-smooth domain $\Oz$, limit of the rough geometry, when the roughness size $\epsilon$
goes to zero. For a fixed $\epsilon$, this approach allows to estimate
the error of an effective wall law approximation defined
only in the smooth domain. In \cite{BrBoMi},  we did
not obtain optimal estimates in the $L^2$  norm. The major difficulty 
was some dual norm of a normal derivative as
explained below. The present work fills this gap.

In section \ref{Section.Framework}, a short presentation of the problem and the material
 introduced in \cite{BrBoMi} are presented.  
The 
 difficulties that this paper overcomes are then faced in the next sections:
firstly, the microscopic approximations live on unbounded
domains and thus belong to weighted Sobolev spaces.
As a result, one needs to derive very weak solutions on a quarter-plane, 
in these spaces (see section \ref{vws}). Then one should connect
these microscopic very weak estimates to the macroscopic
problem we are really interested in. At this scale, the approximations
live in the bounded domain $\Oz$ and 
regular solutions belong to  a specific subspace of $H^1(\Oz)$.
 While this correspondence was introduced  for fractional test spaces
in \cite{BrBoMi}, here it is extended
to the trace spaces specific to the regular solutions above (section \ref{micro-macro}). In section \ref{section.vws}, we analyse the convergence
of the full boundary layer approximation towards the
exact solution using arguments introduced in the previous 
sections; optimal estimates are obtained.
Then for the first order wall-law, the convergence rate is shown to be equal to the one  
 obtained in the periodic case \cite{BrMiQam}. In a last part,
we provide a numerical validation of the theoretical results.
%%Using {\tt freefem++}, we construct boundary layers which may or not include
%%vertical correctors. 
We compare various multi-scale approximations with a numerical solution of 
the complete rough problem. This comparison is made in  Sobolev norms for various values of $\epsilon$. 
An accurate control of the mesh-size with respect to $\epsilon$ and a $\PP_2$ Lagrange finite element provide 
twofold results: the full boundary layer approximation shows the maximal convergence rate that one can 
expect from our numerical discretization, however, the standard averaged wall-law shows poorer results than expected. %%Up to now, a clear understanding of the latter observation has not been provided.

\section{The framework}\label{Section.Framework}

\subsection{The rough domain}

We set a straight horizontal domain $\Oe$, defined
by 
\begin{equation}\label{domain}
\Oe:=\left\{ x \in \RR^2 \text{ s.t. } x_1 \in ]0:1[ \text{ and } \epsilon f\left(\frac{x_1}{\epsilon}\right)<x_2<1\right\},
\end{equation}
where  $f$ is a Lipschitz continuous function, 1-periodic. 
Moreover we suppose that $f$ is bounded and negative definite, i.e. there exists a
positive constant $\delta$ such that $1-\delta < f(y_1)<\delta $ for all $y_1\in [0,2\pi]$.
The lateral 
boundaries are denoted by ${\Gamma_{\rm in}}$ and $\Gout$, and their restrictions to $]0,1[$,
${\Gamma_{\rm in}}p$ (resp. $\Goutp$).  The rough bottom of the domain  is called
$$
\Geps:=\left\{ x\in \RR^2 \text{ s.t. } x_2=\epsilon f\left(\frac{x_1}{\epsilon}\right)\right\},
$$ while the top is smooth and denoted by $\Gun$. In the interior of the domain one
sets the square piecewise-smooth domain $\Oz:=]0,1[^2$, whose lower interface
is denoted by $\Gz$, (see fig.\ref{Plaque}). 
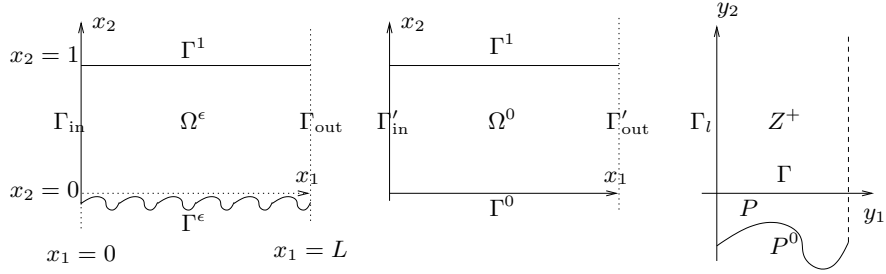
\begin{figure}[h]
\begin{center}
\input{fondrugueux}
\caption{\em (macroscopic) Rough, smooth and (microscopic) cell domains} \label{Plaque}
\end{center}
\end{figure}

\subsection{The exact rough problem}

In order to identify more precisely the influence 
of vertical non-periodic boundary conditions, we consider a singular perturbation
of a linear profile: we look for solutions of the  problem,
find $u\in H^1(\Oe)$ such that
\begin{equation}
\label{RugueuxCompletNP}
\left\{
\begin{aligned}
& - \Delta \uenp = 0,\text{ in  } \Omega^\epsilon, \\
& \uenp = \ovu ,\text{ on } \Gamma^1 , \uenp = 0 \text{ on } \Gamma^\epsilon, \\
& \ddn{\uenp}= 0,\text{ on } \gio.
\end{aligned}
\right.
\end{equation}

When $\epsilon$ goes to zero we recover the linear profile $u^0=\ovu x_2$,
while this profile is explicit what follows shall apply 
with few modifications to the case of an implicit function $u^0$
which  solves the  problem:
$$
\left\{
\begin{aligned}
& - \Delta \uz = 0,\text{ in  } \Omega^\epsilon, \\
& \uz = \ovu ,\text{ on } \Gamma^1 , \uz = 0 \text{ on } \Gamma^\epsilon, \\
& \ddn{\uz}= 0,\text{ on } \giop.
\end{aligned}
\right.
$$
One can show (see \cite{JaMiJDE.01, BrMiQam}) that 
$$
\nrm{\uenp - \uz}{L^2(\Oz)} \leq k\, \epsilon, 
$$
within the  very weak solution framework {\em {\`a} la } Ne{\v c}as, that
will be detailed below (see section \ref{vws}).

\subsection{First order approximation}
%\subsubsection{First order microscopic periodic cell problem}

When one wants to improve the accuracy of the zero order approximation,
one extends  $\uz$ linearly using a Taylor formula in the neighbourhood
of the fictitious interface $\Gz$. So  we have $u^0=\ovu x_2$ for every 
$x$ in $\Oe$. As the Dirichlet condition is no more satisfied on $\Geps$,  
one should solve a microscopic problem that reads: find $\beta$, whose 
Dirichlet norm is finite, such that
\begin{equation}
\label{A.cell}
\left\{
\begin{aligned}
& -\Delta \beta = 0,\text{ in } \zup,\\
& \beta = - y_2,\text{ on } P^0,\\
& \beta \text{ is } \yup.
\end{aligned}
\right.
\end{equation}
where $Z^+:=]0,1[\times\RR_+$, $P:=\{y \in \RR^2\text{ s.t. }y_1 \in ]0,1[, \;\, f(y_1)<y_2<0\}$ and
$P^0:=\{y \in \RR^2 \text{ s.t. } y_1 \in ]0,1[, \;\; y_2=f(y_1)\}$ (see fig. \ref{Plaque} right). In the literature 
this problem is widely studied (see \cite{JaMiPise.96,NeNeMi.06,BrMiQam}),
so we only sum up the main properties of $\beta$.
\begin{lem} There exists a unique solution of problem \eqref{A.cell}.
Moreover, 
$$
\lim_{y_2\to \infty}\beta(y_1,y_2)=\obeta
\text{ for every }y_1, \text{ and }
\obeta := \int_{(0,1)} \beta(y_1,0) dy_1.
$$
The convergence is exponential and one has a Fourier decomposition:
$$
\beta(y)=\sum_{k=-\infty}^\infty \beta_k e^{2\pi(-|k|y_2+iky_1)}, \, \forall y \in Z^+, \text{ where } \beta_k := \int_0^1 \beta(y_1,0) e^{2\pi i k y_1} dy_1.
$$  
\end{lem}

If $\uenp$ were periodic, we could set the first order approximation to be
\begin{equation}\label{uiuep}
\uiuep := \uz + \left(\frac{\epsilon}{1+\epsilon \obeta}\right) \dd{\uz}{x_2}(x_1,0)\left( \beta\lrxe - \obeta x_2 \right).  
\end{equation}
But  this does not satisfy the homogeneous Neumann boundary conditions on $\gio$ when approximating 
the solution of \eqref{RugueuxCompletNP}. In \cite{BrBoMi} we introduced
a vertical corrector. We denote it by $\tin$; it solves the problem:
\begin{equation}
\label{QuaterPlane}% ({\cal P})
\left\{
\begin{aligned}
& -\Delta \tin = 0, \quad \text{ in } \Pi, \\
& \ddn{\tin}(0,y_2) = - \ddn{ \beta }(0,y_2) ,\quad \text{ on } E,\\
& \tin = 0, \quad  \text{ on } B.
\end{aligned}
\right.
\end{equation}
where we set
$\Pi := \cup_{k=0}^{+\infty} [ \zup + k \eu]$. The vertical boundary is denoted by $E:=\{y\in \Pi, y_1=0\}$ and  the bottom by $B:=\cup_{k=0}^{+\infty} \{y\in P^0 \pm k\eu\}$ (cf. fig \ref{pi}).
In what follows we will write $\Pi':=\rr^2$, $B':=\rr\times \{0\}$ and $E':=\{0\}\times \rr$.
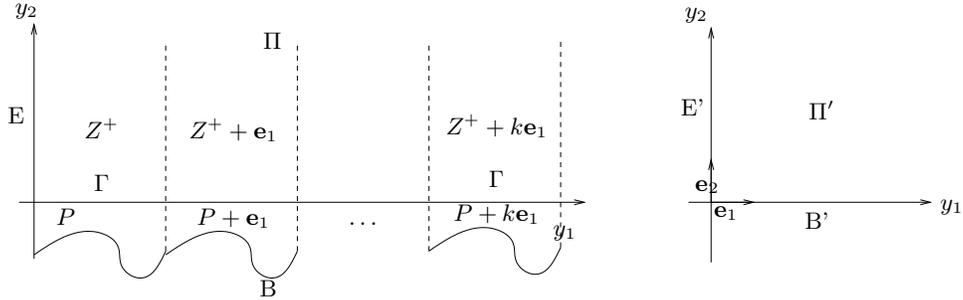
\begin{figure}[h!] 
\begin{center}
\input{fond_qp}
\caption{\em Semi infinite microscopic domains: $\Pi$, the rough quarter-plane and $\Pi'$, the smooth one } \label{pi}
\end{center}
\end{figure} 
For the rest of the paper, we define the usual Sobolev space:
$$
%\begin{aligned}
\ws{m}{p}{\alpha}{\Omega}:=\left\{ v \in \cD'(\Omega) \, \text{ s.t. } \, |D^\lambda v| (1+\rho^2)^{\frac{ \alpha + |\lambda| - m}{2}{\nobreakspace}} \in L^p(\Omega),\, 0\leq |\lambda| \leq m \,\right\}
%\end{aligned}
$$
where $\rho$  is a distance to the point (0,-1) exterior to the domain $\Pi$. Shifting the latter
point to $(0,0)$ gives an equivalent norm so that we will not distinguish between these two distances.
We refer to \cite{Ha.71,Ku.80.book,AmGiGiI.94} and references therein, for the detailed studies of the weighted Sobolev
spaces in the context of elliptic  operators.

In the first part of this study \cite{BrBoMi}, we have rigorously shown the  results
regarding $\tin$:
\begin{thm}\label{BrBoMi}
There exists a unique solution $\tin \in \ws{1}{2}{\alpha}{\Pi}$ of problem \eqref{QuaterPlane} where $\alpha$ 
is such that $|\alpha|<\alpha_0:=\sqrt{2}/ \pi$, moreover 
$$
|\tin(y)| \leq \frac{k}{(1+\rho^2(y))^{\ud\left(1-\frac{1}{2M}\right)}}, \quad  \forall y \in \Pi \text{ s.t. }\rho(y)>1
$$
where $M$ is a positive constant such that $M<1/(1-2\alpha)\sim10$.
\end{thm}
This theorem is based on Poincar{\'e}-Wirtinger estimates for the Sobolev part
and a Green's representation formula in a quarter-plane $\Pi'$.  Combining this two arguments, one obtains estimates on 
the decreasing properties of $\tin$. In the same way we define the vertical boundary layer corrector for $\Gout$ that solves the problem:
\begin{equation}
\label{QuaterPlaneLeft}% ({\cal P})
\left\{
\begin{aligned}
& -\Delta \tout = 0, \quad \text{ in } \Pi_-, \\
& \ddn{\tout}(0,y_2) = - \ddn{ \beta }(0,y_2) ,\quad \text{ on } E,\\
& \tout = 0, \quad  \text{ on } B_-.
\end{aligned}
\right.
\end{equation}
where we set
$\Pi_- \equiv \cup_{k=1}^{+\infty} [ \zup - k \eu]$, the bottom being denoted by $B_-=\cup_{k=1}^{+\infty} \{y\in P^0 - k\eu\}$.  In Theorem, \ref{BrBoMi}, as  everywhere else in the rest of the paper, the properties derived for $\tin$ are equally valid for $\tout$. Thanks to these correctors, one completes
the previous boundary layer approximation by writing:

\begin{equation}\label{fbla}
\uiue := \uz 
      + \left(\frac{\epsilon}{1+\epsilon \obeta}\right) \dd{\uz}{x_2}(x_1,0)\left( \beta\lrxe - \obeta x_2 
								  + \tin\lrxe + \tout\left(\frac{x_1-1}{\epsilon},\frac{x_2}{\epsilon}\right)
	\right).
\end{equation}
This approximation satisfies the  problem:
$$
\left\{ 
\begin{aligned}
&- \Delta \uiue = 0 , \text{ on } \Oe,\\
& \ddn{\uiue} =   \ddn{\tin}\left(\frac{1}{\epsilon},\frac{x_2}{\epsilon}\right) \text{ on } \Gout ,\quad \ddn{\uiue} =   \ddn{\tout}\left(\frac{1}{\epsilon},\frac{x_2}{\epsilon}\right) \text{ on } {\Gamma_{\rm in}}, \\%\quad 
& \uiue = \ovu + \epsilon  \dd{u^1 }{x_2}(x_1,0)\left( \left( \beta - \obeta +\tin \right) \left(\frac{x_1}{\epsilon},\frac{1}{\epsilon} \right)+  \tout\left(\frac{x_1-1}{\epsilon},\frac{1}{\epsilon}\right)\right) \text{ on } \Gun, \\
&  \uiue = 0 , \text{ on } \Geps 
\end{aligned}
\right.
$$
In other words, we correct the $O(1)$ error on the normal derivatives of the boundary layer $\beta$ on $\gio$
by the normal derivatives of the correctors at a distance $y_1=\ue$ of the vertical 
boundary $E$. On $\Gun$, the main errors are due to  $\tin$ and $\tout$, the 
contribution of $\beta(x/ \epsilon)-\obeta$ being exponentially small on $\Gun$.

In \cite{BrBoMi} we set up adequate tools to handle 
{\em a priori} estimates for the error. We adapt them 
to the specific boundary conditions in problem 
\eqref{RugueuxCompletNP}, and claim
\begin{thm}\label{a.priori}
 The boundary layer approximation $\uiue$ satisfies
the  error estimates in the Dirichlet norm:
$$
\nrm{\nabla (\uenp - \uiue ) }{L^2(\Oe)}\leq k \epsilon ^{\min \left( 1+\alpha ; \td-\frac{1}{2M} \right)}
$$
where the constant $k$ is independent of $\epsilon$ and the constants $\alpha$ and $M$ 
are defined as in Theorem \ref{BrBoMi}.
\end{thm}
While in Theorem 5.1 \cite{BrBoMi}, the convergence order is only $\epsilon$, here, 
it is improved because the perturbed profile is only
linear: there are no second order errors in the
sub-layer. The proof follows exactly the same ideas 
as in Theorem 5.1 in \cite{BrBoMi}. The estimates 
for the very weak solutions presented in Theorem 5.2 in \cite{BrBoMi}
are not optimal and this is the main concern of the present
article.

\section{Dirichlet-Poisson problem in a quarter-plane}
\label{vws}
\paragraph{Domains, coordinates and notations}

We define the shifted domain $\ppl:=]l,\infty[\times \RR_+$ and its boundary 
$\dppl:=E'_l\cup B'_l$ where $E'_l:=\{ l \} \times \RR_+$ and $B'_l:=]l,\infty[\times\{ 0\}$.
We denote by $\pplt$ the $\pi/4$ radians rotation of $\ppl$ with respect to  the origin $(0,0)$
and $\dpplt$ the corresponding boundary. $\rrl:=\RR\times ]l',\infty[$ will represent
the straightening of  $\pplt$, and $l':=l/ \sqrt{2}$. The domain $\pplt$ can be parametrised 
as $\pplt =\{ z \in \RR^2 \text{ s.t. }   z_1 \in \RR \text{ and } z_2>l'+|z_1-l'|=:a(z_1) \}$
whereas the change of variables from $\pplt$ towards $\rrl$ is given by $w=\cW(z)$ s.t.
$$
\left\{ 
\begin{aligned}
w_1&=z_1,\\
w_2&=z_2-|z_1-l'|.
\end{aligned}
\right.
$$
Later on we will also need the regularised version of
 domains above 
\begin{equation}\label{affine}
\begin{aligned}
\ppls&:=\left\{ y \in \ppl  \text{ s.t. } \quad y_2 > \frac{s}{y_1-l}\right\}, \\
\pplst& := \left\{ z \in \pplt \text{ s.t. } \quad z_2 > l'+\sqrt{ (z_1-l')^2 +s^2 }=:a_s(z_1) \right\} ,
\end{aligned}
\end{equation}
and the corresponding boundaries are set according to this definition. The mapping straightening $\pplst$ to $\rrl$ is set as $w=\cWs(z)$:
$$
\left\{ 
\begin{aligned}
w_1&=z_1,\\
w_2&=z_2-\sqrt{(z_1-l')^2 +s^2}.
\end{aligned}
\right.
$$

\subsection{Weak solutions}%\subsection{Weak solution in the Weighted context}

We consider the solution of the  problem: find $v$ 
in $\ws{1}{2}{0}{\ppl}$ solving
\begin{equation}
\label{qpt}
\left\{ 
\begin{aligned}
& \Delta v = 0, \quad \text{ in } \ppl, \\
& v = g, \quad \text{ on } \dppl,
\end{aligned}
\right.
\end{equation}
where $g$ is a function belonging to $\ws{1}{2}{\ud}{\dppl}$.
We emphasise here that, as in Chap. 5 \cite{Ne.Book.67},  we require a little 
more regularity on $g$ than the usual fractional norm. As we work
with  weighted Sobolev spaces and we control
the tangential derivatives of the data, the existence and uniqueness results
for weak solutions are nor standard neither so straightforward. As they 
will be used extensively in what follows we provide  a detailed 
presentation.

In order to give a variational formulation of problem \eqref{qpt},
we need to construct lifts of the boundary data in the weighted Sobolev
context. As we apply  changes of variables above, we have to insure the 
compatibility of weights with respect to these mappings. Thus 
we present a detailed adaptation of the results for
the half plane introduced in \cite{Ha.71}.

To solve problem \eqref{qpt} we use the Poincar{\'e}-Wirtinger estimates
because the weighted logarithmic Hardy estimates are not valid
in the specific case where $\alpha + n/p$ is an integer in $\ws{n}{p}{\alpha}{\ppl}$
(see \cite{AmGiGiI.94,AmGiGiII.97} and references therein).

\subsubsection{Weighted Sobolev extensions}

The first change of variables presented above, $z(y)$, is
a rotation of the domain around the origin: it 
preserves the distances to the origin.
Once straightened in $\rrl$ we are exactly in the
position to construct extensions introduced by  Hanouzet in Theorem II.2 of \cite{Ha.71},
so we set:
$$
\left\{ 
\begin{aligned}
\Psi(w)&= \Phi \left( \frac{w_2-l'}{\sqrt{1+w_1^2+(l')^2}}\right), \quad \forall \, w \in \rrl \\
V (w)&= \int_{|t|<1} \hR(t) g(t(w_2-l')+w_1) dt,
\end{aligned}
\right.
$$
where $\Phi$ is a cut-off function such that
$$
{\rm Supp} \Phi \in [0,\uqt[,\quad \Phi(0)=1,\quad \Phi \in C^\infty([0,\uqt]),
$$
and $\hR$ is a regularising kernel i.e. $\hR \in C^\infty_0(]-1:1[)$
and $\int_\R \hR(s) ds=1$.
Our lift then reads:
$$
R(g)(w)=\Psi(w)V(w),\quad \forall \,w \in \rrl.
$$
Then one has:
\begin{lem}\label{redresse}
 For every function $g\in \ws{1}{2}{\ud}{\R}$ one has:

\begin{equation}\label{eq.redresse}
\nrm{R(g)}{\ws{1}{2}{0}{\rrl}} \leq k \nrm{g}{\ws{1}{2}{\ud}{\R\times \{ l'\} }},  
\end{equation}
where the constant $k$ is independent on $l'$.
\end{lem}
\begin{proof} 
According to the definition of the norm associated to $\ws{1}{2}{0}{\ppl}$, 
$$
\begin{aligned}
\int_{\rrl} \left( \frac{\Psi V}{\rho} \right)^2 dw & \leq k \int_{\rrl} \left(\frac{\Psi}{\rho} \right)^2 \int_{|t|<1} g^2(t (w_2-l')+ w_1) dt dw\\
& \leq k \int_{\R\times [0,\uqt]} \frac{\sqrt{1+w_1^2+(l')^2}}{1+w_1^2+(l')^2 + x^2(1+w_1^2+(l')^2)} \Phi^2(x)\cdot \\
& \quad\quad \cdot  \int_{|t|<1} g^2(t x\sqrt{1+w_1^2+(l')^2}+ w_1) dt dx dw_1  \\
& \leq k \int_\R \frac{g^2}{\sqrt{1+\ti{w}_1^2+(l')^2}} d\tilde{w}_1 \leq k \nrm{g}{\ws{1}{2}{\ud}{\R}},
\end{aligned}
$$
where we used the change of variables $x=(w_2-l')/\sqrt{1+\ti{w}_1^2+(l')^2}$ and a shift 
$\ti{w_1}:=t x\sqrt{1+w_1^2+(l')^2}+ w_1$ in a second step as suggested in Lemma II.2 \cite{Ha.71}.
If $g$ is a  regular function, the gradient is estimated as:
$$
\snrm{\nabla(\Psi V)}{} \leq \frac{ \left| ( \nabla \Psi )V\right|}{\sqrt{1+w_1^2+(l')^2}}  + | \Psi \nabla V | , \quad \forall w \in \rrl \, \text{ s.t. } \, w_2 \leq \uqt \sqrt{1+w_1^2+(l')^2}
$$
So  the first integral proceeds as above, whereas in the second part,
performing the same changes of variables, one gets:
$$
\begin{aligned}
\int_{\rrl} \Psi^2 |\nabla V| dw & \leq k \int_{\rrl} \Psi^2\int_{|t|<1} g'(t (w_2-l') + w_1)(1+t)^2 dt dw \\
&\hspace{-1.4cm} \leq k \int_{\R\times [0,\uqt]} {\sqrt{1+w_1^2+(l')^2}} \Phi^2 (x)%% \cdot \\ & \quad\quad \cdot  
\int_{|t|<1} \left(g' \left(t x \sqrt{1+w_1^2+(l')^2} + w_1 \right) \right)^2dt dx dw_1 \\
& \leq k \int_\R \sqrt{1+\ti{w}_1^2+(l')^2} (g')^2 d\tilde{w}_1 \leq k \nrm{g}{\ws{1}{2}{\ud}{\R}}
\end{aligned}
$$
extending this to $\ws{1}{2}{\ud}{\R}$ functions by density arguments ends the proof.
\end{proof}
In order to use these estimates in $\ppl$ we have to guarantee that they apply also when the domain is a quarter-plane.
\begin{lem}\label{lem.lift.redresse.cont} For every $g\in \ws{1}{2}{\ud}{\dppl}$ there exists a lift denoted by $R(g)$ in $\ws{1}{2}{0}{\ppl}$
such that
$$
\nrm{R(g)}{\ws{1}{2}{0}{\ppl}} \leq k \nrm{g}{\ws{1}{2}{\ud}{\dppl}},
$$
where the constant $k$ is independent on $l$.
\end{lem}
\begin{proof}
As mentioned above the rotation does not change the distances, 
so one has to consider the mapping from $\rrl$ to $\pplt$ (resp. $\RR\times\{l'\}$ to $\dpplt$) from 
both sides of \eqref{eq.redresse}.
The straightening by the continuous piecewise linear transform $w=\cW(z)$ is defined in \eqref{affine}; one has that 
$$
\int_{\rrl}|\nabla_w R(g)| dw = \int_{\pplt} (\nabla_z\cW^{-1} \nabla_z\cW^{-T} \nabla_zR(g), \nabla_zR(g)) \det (\nabla_z \cW) dz.
$$
The eigenvalues of $\nabla_z\cW^{-1} \nabla_z\cW^{-T} $ read $\lambda_\pm=(3\pm\sqrt{5})/2$, they are positive definite 
and independent on $l'$. Thus there exists a constant $k$ such that
$$
\nrm{R(g)}{\ws{1}{2}{0}{\ppl}}\leq k \nrm{R(g)}{\ws{1}{2}{0}{\rrl}}. 
$$
On the other hand, one should focus on the equivalence of trace norms between $\ws{1}{2}{\ud}{\dpplt}$ and $\ws{1}{2}{\ud}{\RR\times \{ l'\}}$:
on $\dpplt$, $1+\rho^2(z)=1+(z_1)^2 + (l'+|z_1-l'|)^2$. This gives the existence of  a constant $k$ independent on $l'$ s.t.
$$
1+z_1^2+(l')^2 \leq  \rho(z) \leq k ( 1 + z_1^2+(l')^2), \quad \forall \, z \in \dpplt
$$
In turn this implies that
$$
\begin{aligned}
&\int_{\dpplt} \left\{ \frac{g^2}{(1+\rho^2(z))^\ud} +  (1+\rho^2(z))^\ud (g')^2 \right\} d\sigma (z) \\
& \quad \quad \sim \int_{\RR\times\{ l'\} }  \left\{ \frac{g^2}{(1+w_1^2+(l')^2)^\ud} +  (1+w_1^2+(l')^2)^\ud (g')^2 \right\} dw_1.
\end{aligned}
$$
\end{proof}

\begin{rmk}
The previous lemma applies with only minor changes to the case of the smooth domain sequence $(\ppls)_{s\in]0,1]}$.  
\end{rmk}
\subsubsection{{\em A priori} estimates}
At this point we are ready to prove the existence and uniqueness of a weak solution
of problem \eqref{qpt}. We denote by $\wso{1}{2}{0}{\ppl}$ the 
subspace of $\ws{1}{2}{0}{\ppl}$ such that:
$$
\wso{1}{2}{0}{\ppl}=\{ v \in \ws{1}{2}{0}{\ppl}\quad \text{ s.t. }\quad v= 0 \text{ on } \dppl\}. 
$$

\begin{proposition}\label{prop.exist.weak}
There exists a unique weak solution $v\in \ws{1}{2}{0}{\ppl}$, of the problem \eqref{qpt}, moreover one has:
$$
\nrm{v}{\ws{1}{2}{0}{\ppl}}\leq k \nrm{g}{\ws{1}{2}{\ud}{\dppl}},
$$
where the constant $k$ is independent of $l$.
\end{proposition}

\begin{proof}
Using the lift of $g$ given above,  \eqref{qpt} becomes: find $\tv\in \wso{1}{2}{0}{\ppl}$ s.t. $\Delta \tv=-\Delta R(g)$. Testing this equation by $\varphi\in\D(\ppl)$ one has that
$$
(\nabla \tv,\nabla\varphi )=-(\nabla R(g),\nabla \varphi),
$$
by density of $\D(\ppl)$ functions in $\wso{1}{2}{0}{\ppl}$, the r.h.s. is a linear
form on  $\wso{1}{2}{0}{\ppl}$ and the l.h.s. is a bi-linear bi-continuous form on the same
functional space. Thanks to Poincar{\'e}-Wirtinger estimates in a quarter-plane 
the semi-norm is actually equivalent to the $\ws{1}{2}{0}{\ppl}$ norm.
Thus one has the existence and uniqueness by the standard Lax-Milgram theorem. 
Moreover, one has that
$$
\nrm{\tv}{\ws{1}{2}{0}{\ppl}} \leq k \nrm{R(g)}{\ws{1}{2}{0}{\ppl}} \leq k'  \nrm{g}{\ws{1}{2}{\ud}{\dppl}},
$$
where all the constants do not depend on $l$.
Subtracting $R(g)$ one gets the desired result.
\end{proof}

\subsubsection{Regularised problems}
In the rest of the paper, we need a little more 
regularity in order to construct a Green's formula
adapted to the Lipschitz domain $\ppl$ 
(see the very clear and detailed explanations of
{\S} 1.5.3 in \cite{Grisvard}). Thus  in this paragraph, we construct 
regular approximations of problem \eqref{qpt}. This 
is done by approximating $\ppl$ by a sequence $(\ppls)_{s\in [0,1]}$
of $C^\infty$ domains defined above.

For every given function $g\in \ws{1}{2}{\ud}{\dpplt}$, one
sets $g_s \in   \ws{1}{2}{\ud}{\dpplst}$ by writing
$$
g_s(z_1,a_s(z_1)):=g(z_1,a(z_1))=g(z_1), \quad \forall z_1 \in \R
$$
and, where it is not ambiguous, we will drop the $s$ 
and use $g$ instead of $g_s$. 

\begin{lem}\label{approx}
  The approximating sequence of data $(g_s)_{s\in [0,1]}$ is stable
with respect to  the $\ws{1}{2}{\ud}{\dpplt}$ norm:
$$
\nrm{g_s}{\ws{1}{2}{\ud}{\dpplst}} \leq k \nrm{g}{\ws{1}{2}{\ud}{\dpplt}},
$$
whereas for the lifts one has
$$
\nrm{R_s(g_s)}{\ws{1}{2}{0}{\pplst}} \leq k' \nrm{g}{\ws{1}{2}{\ud}{\dpplt}},
$$
and the sequence $(R_s(g_s))_s$ converges to $R(g)$ in the $ \ws{1}{2}{0}{\pplt}$ norm:
$$
\forall \eta \quad \exists \delta > 0 \quad \text{ s.t. } \quad 0<s<\delta \; \implies \;\nrm{R_s(g_s)-R(g)}{ \ws{1}{2}{0}{\pplt}} \leq \eta.
$$
\end{lem}
\begin{proof}
For $s$ small enough, one has:
$$
a(z_1)\leq a_s(z_1)\leq 2a(z_1),\quad \forall z_1 \in \RR.
$$
Similarly it is easy to show that $|a'_s|\leq |a'|$. This gives the first claim. 
The proof of the second claim is identical the proof of Lemma \ref{lem.lift.redresse.cont}.
We extend $R_s(g_s)$ in $\pplt \setminus \pplst$ by $g$. It is easy to 
prove that there exists a constant $k$ s.t. 
$$
\nrm{R_s(g_s)}{\ws{1}{2}{0}{\pplt}}\leq k \nrm{g}{\ws{1}{2}{\ud}{\dpplt}}.
$$
Up to a subsequence, the compact imbedding of $\ws{1}{2}{0}{\pplt}$ into $\ws{0}{2}{-1}{\pplt}$
implies strong convergence in the latter norm. The main focus is the semi-norm 
convergence. As seen in the lemma above the norms are consistent when passing
from $\pplt$ to $\rrl$, and the same holds when passing from $\pplst$ to 
$\rrl$ for the same reason. Making the change of variables $\cWs$, one
can re-express both lifts in $\rrl$ as:
$$
\left\{ 
\begin{aligned}
R_s(g)&=\Phi\left(  \frac{w_2-l'}{\sqrt{1+w_1^2+(l')^2}}\right) V(w), \\
R(g) &=\Phi\left(  \frac{w_2+\oms(w_1)-l'}{\sqrt{1+w_1^2+(l')^2}} \right) V(w_1,w_2+\oms(w_1)), \\
\end{aligned}
\right.
$$
where $\oms(w_1)=\sqrt{s^2+w_1^2}-|w_1|$. For the rest of the proof we set 
$\tw=(w_1,w_2+\oms(w_1))$. We decompose $I:= \nabla_w(R_s(g)-R(g))$ in four pieces:
$$
\begin{aligned}
I:= \sum_{j=1}^4I_j := & \nabla_w( \Psi(\tw) - \Psi (w))V(\tw) + \nabla\Psi(w) (V(\tw)-V(w)) \\
	& + (\Psi(\tw)-\Psi(w))\nabla V(\tw) + \Psi(w)\nabla(V(\tw)-V(w)).   
\end{aligned}
$$
The first three terms can easily be estimated by $s^\gamma \nrm{g}{\ws{1}{2}{\ud}{\RR}}$
where $\gamma$ is some positive constant, thanks to techniques  used in \cite{Ha.71}
 for fractional trace spaces. In our case those terms are even easier to 
treat because no fractional norm has to be used. The term $I_4$ is more delicate
because there is no derivative left, so we have to  use the continuity of weighted translation
operators; here we give  the sketch of the proof:
$$
\begin{aligned}
I_4 & \leq 2 \left| \Psi(w) \int_{|t|<1} \hR(t) \left\{g'(t\tw_2 +w_1)-g'(tw_2+w_1)\right\}dt \right| \\
& \quad \quad + \left| \Psi(w) \int_{|t|<1} \hR(t) {g'(t\tw_2 +w_1) t \oms'(w_1)} dt \right| %% \\
 =: I_{4,1}+I_{4,2},
\end{aligned}
$$
the term $I_{4,2}$ is estimated again as the other terms, we focus on $I_{4,1}$
$$
\begin{aligned}
J:=\int_{\rrl}I_4^2 dw \leq 2 k \int_{\RR\times[0,\uqt]} \Phi^2 (x) & \left\{ g'\left(tx\sqrt{ 1 + w_1^2 + (l')^2} + \oms(w_1) + w_1\right) \right.\\
&\quad  \left.-g'\left(tx \sqrt{ 1 + w_1^2 + (l')^2} + w_1\right) \right\}^2 dt,
\end{aligned}
$$
where, as before, we set $x:= (w_2-l')/\sqrt{ 1 + w_1^2 + (l')^2}$. Then we use a
version of Lemma II.2. in \cite{Ha.71} extended to all types of powers of the integration weight to conclude 
that:
$$
J \leq k\int  \sqrt{ 1 + \tw_1^2 + (l')^2} \left\{ g'( \tw_1 + tz \tilde{\omega}(t,z,\tw_1))- g'( \tw_1 ) \right\}^2 dt dz d\tw_1,
$$
where $\ti{\omega}(t,z,\tw_1)=\omega(w_1(t,z,\tw_1))$.
At this point, we can follow the proof of Theorem 2.1.1 in \cite{Ne.Book.67} that 
states the continuity of the translation operator in $L^p$.
The only difference is that the term $ tz \tilde{\omega}(t,z,\tw_1))$ itself depends on the 
integration variables. This dependence problem is overcame by noting that
$$
\forall \eta > 0,\, \exists \delta > 0\,  \text{ s.t. }\, |tz  \tilde{\omega}(t,z,\tw_1))| \leq \frac{s}{4} < \delta,\quad \forall (t,z,\tw_1) \in ]-1,1[\times\left[0,\uqt\right[\times\RR,
$$
this in turn implies that on the set of continuity points of $g'$, one has
$$
\left|g'(\tw_1 + tz \tilde{\omega}(t,z,\tw_1))- g'( \tw_1 )\right| < \frac{\eta}{3},
$$
and the rest follows exactly as in Theorem 2.1.1 in \cite{Ne.Book.67}.
\end{proof}
\subsubsection{Convergence}
Now we state the existence and uniqueness of the regularised problem: find $\vs \in \ws{1}{2}{0}{\dppls}$ s.t.  
\begin{equation}\label{pqt.reg}
\left\{ 
\begin{aligned}
& \Delta \vs = 0, \text{ in } \ppls,\\  
& \vs = g, \text{ on } \dppls.
\end{aligned}  
\right.
\end{equation}
\begin{proposition}\label{prop.conv.sol.reg}
For every fixed $s\in[0,1]$   there exists a
unique weak solution $\vs \in \ws{1}{2}{0}{\ppls}$ of problem \eqref{pqt.reg}, satisfying
$$
\forall \eta>0 \quad  \exists \delta>0 \quad \text{ s.t. } \quad s<\delta \Longrightarrow \nrm{\vs -v }{\ws{1}{2}{0}{\ppl}} < \eta  ,
$$
where $v_s$ is extended by $g$ in $\ppl \setminus \ppls$.
\end{proposition}
The proof is a straightforward consequence of Proposition \ref{prop.exist.weak} applied to the regular domain $\ppls$ and of Lemma \ref{approx} for the convergence part.

If we now restrict ourselves to the case where $g\in {\cal E}(\ppl)$, this latter space being dense in $\ws{1}{2}{\ud}{\ppl}$ (see for instance Theorem I.1 in \cite{Ha.71}), we get more regularity, namely:
\begin{lemma}\label{lem.reg}
If $g\in {\cal E}(\dppl)$ then $\vs$, the unique solution of problem \eqref{pqt.reg}, belongs to
$H^2_{\loc}(\ov{\ppls})$ for every fixed $s\in ]0,1]$.
\end{lemma}
\begin{proof}
It is easy to show that if $g\in {\cal E}(\dppl)$, then $R_s(g) \in  {\cal E}(\ppl)$ and
thus by standard interior regularity results one gets that $\tvs:=\vs - R_s(g)$, which belongs
to $\wso{1}{2}{0}{\ppls}$, is actually in $H^2_\loc(\ppls)$ \cite{Evans.Book,GiTru.Book,Ne.Book.67}. 
Moreover because of the Dirichlet
condition and the $C^\infty$ regularity of the boundary, the regularity of $\tvs$ can be extended up to the boundary
by the same method. 
\end{proof}
\subsection{Weighted Rellich estimates}
Above, we constructed  the tools necessary to adapt the
very weak solutions presented in Chap. 5 in \cite{Ne.Book.67}, to the weighted context.

\begin{proposition}\label{rellich}
Let $s\in]0,1]$ be fixed. If $g\in  {\cal E}(\ppls)$ then $\vs \in H^2_\loc(\ppls)$
and one has  moreover
$$
\nrm{\ddn{\vs}}{\ws{0}{2}{\ud}{\dppls}} \leq k \nrm{g}{\ws{1}{2}{\ud}{\dppls}}.
$$
The operator $T$ defined from ${\cal E}(\ppls)$ as 
$T(g)=\ddn{\vs}$ is extended by continuity into a mapping 
from $\ws{1}{2}{\ud}{\dppls}$ on $\ws{0}{2}{\ud}{\dppls}$
and one has that 
$$
\nrm{T(g)}{\ws{0}{2}{\ud}{\dppls}} \leq k \nrm{g}{\ws{1}{2}{\ud}{\dppls}},
$$
where the constants $k$ do not depend on $l$.  
\end{proposition}
\begin{proof}
We rotate again $\ppls$ by $\pi/4$ radians to switch to the chart $(z_1,z_2)$:
  the boundary of $\dpplst$ is expressed as $(z_1,a_s(z_1))$. We set the partition
of unity, 
$$
\sum_{r=0}^{+\infty}\varphi_r(z)=1,  \quad \forall z \in \ppls, 
$$
the functions $\varphi_r$ being defined as:
\begin{equation}\label{cutof}
\left\{  
\begin{aligned} 
&\psi_r:= e^{-\left(\frac{\rho-2^{r}}{\rho-2^{r-1}}\right)^2} \chiu{[2^{r-1},2^r]}(\rho) 
%+\\
%&
%\quad \quad \quad \quad \quad \quad   
+\chiu{[2^r,2^{r+1}]}(\rho) +e^{-\left(\frac{\rho-2^{r+1}}{\rho-2^{r+2}}\right)^2} \chiu{[2^{r+1},2^{r+2}]}(\rho)   , \, r\geq 1, \\
&\psi_0 :=  \chiu{[0,2]}(\rho) +e^{-\left(\frac{\rho-2}{\rho-4}\right)^2} \chiu{[2,4]}(\rho),\\
&\varphi_r:= \frac{\psi_r}{\sum\limits_{j=r-1}^{r+1} \psi_j},\quad \forall r\geq1,\quad \varphi_0 := \frac{\psi_0 }{\psi_0+\psi_1},
\end{aligned}  
\right. 
\end{equation}
where by $\chiu{S}$ we denote the characteristic function of a given set $S$.
Then we define
$$
h^r := (0,-\varphi_r(z) \rho(z)),\quad \text{and } h:=\sum_{r=0}^{+\infty} h^r.
$$
Thanks to Proposition \ref{lem.reg}, $\vs \in H^2_\loc(\ov{\pplst})$; one is allowed  to set locally the Rellich formula (\cite{Ne.Book.67}, p. 245). When adapted to the Laplace operator, it reads:
$$
\begin{aligned}
&   \int_{\dpplst} ((h^r\cdot \nu)\, \id - ( h \otimes \nu + \nu \otimes h ) \nabla\vs, \nabla \vs)\, d\sigma (z) \\
& = \int_{\pplst}^{}(\dive h^r \, \id -(\nabla h^r + (\nabla h^r)^T) \nabla \vs, \nabla \vs ) \,dz% \\
+ \int_{\pplst}^{}(h^r\cdot \nabla \vs) \Delta \vs \,dz.
\end{aligned}
$$
We have that 
\begin{equation}\label{est.below}
 (\nu,h^r)=\varphi_r(z)(1+(a_s')^2)^{-\ud} \rho(z) \geq k \varphi_r \rho(z)   
\end{equation}
because $|a_s'|<1$. 
Developing the boundary term in normal ($\nu$) and tangent ($\tau$) 
directions, one has that, in fact,
$$
\begin{aligned}
&  \int_{\dppls}^{}(h^r\cdot\nu)(\partial_\nu \vs)^2 + 2(h^r\cdot\tau) \partial_\tau\vs \partial_\nu \vs  - (h^r\cdot\nu)(\partial_\nu \vs)^2  \,d\sigma(z) \\
& =- \int_{\pplst}^{}(\dive h^r \, \id -(\nabla h^r + (\nabla h^r)^T) \nabla \vs, \nabla \vs ) \,dz. 
\end{aligned}
$$
The first term in the r.h.s. above is estimated from below thanks to \eqref{est.below}, the second and the third ones by their absolute value, giving:

$$
\begin{aligned}
&  \int_{\dppls}^{}\varphi_r(z)\rho(z) (\partial_\nu \vs)^2  \,d\sigma(z) \leq %%\\
%%& 
\int_{\dppls}^{}2\varphi_r\rho(z)( |\partial_\tau\vs| |\partial_\nu \vs|  + (\partial_\nu \vs)^2 )  \,d\sigma(z) \\
&\quad \quad  \quad \quad   - \int_{\pplst}^{}(\left\{ \dive h^r \id -(\nabla h^r + (\nabla h^r)^T) \right\} \nabla \vs, \nabla \vs ) \,dz. 
\end{aligned}
$$
Then we sum with respect to  $r$ and
apply the Beppo-Levi theorem for the boundary terms.
For the interior r.h.s. above, due to the specific choice of cut-of function
in \eqref{cutof}, one can pass to the limit with respect to  the summation index $r$
applying the Lebesgues theorem.
These justifications allow us to write
$$
\begin{aligned}
\int_{\pplst}^{}\rho(z)(\partial_\nu \vs)^2 \,d\sigma& (z) \leq   \int_{\dpplst}^{}2\rho(z)  \left( | \partial_\nu\vs  | | \partial_\tau \vs | + (\partial_\nu\vs)^2  \right)  \,d\sigma (z) \\
& - \int_{\ppls} \left(\left\{ \dive\begin{pmatrix} 0 \\ \rho \end{pmatrix} \id - (\nabla+\nabla^T)\begin{pmatrix} 0 \\ \rho \end{pmatrix}  \right\} \nabla \vs , \nabla \vs \right) dz,
\end{aligned}
$$
note that it is important here that the derivatives of $h$ contain
only $\rho$ but no cut-of function, this explains why we don't
estimate the interior terms before summing over $r$. 
By Cauchy-Schwartz one gets that
$$
\nrm{\ddn{\vs}}{\ws{0}{2}{\ud}{\dpplst}}^2 \leq k \left\{ \nrm{g}{\ws{1}{2}{\ud}{\dpplst}}\nrm{\ddn{\vs}}{\ws{0}{2}{\ud}{\dpplst}} + \nrm{\vs}{\ws{1}{2}{0}{\pplst}} \right\} 
$$
Thanks to the Young inequality and Proposition \ref{prop.exist.weak}, one obtains the first estimate of the claim.
Because the r.h.s. only depends on the $\ws{1}{2}{\ud}{\dppls}$-norm,  the result can be extended to every function
$g$ belonging to $\ws{1}{2}{\ud}{\dppls}$ by density and continuity.
\end{proof}

\begin{rmk}
Note that this result holds in fact for any polynomial of
$\rho$ and what follows could be extended as well to 
any weighted Sobolev space. One only needs  to choose the
proper scaling for the cut-off functions $\varphi_r$ with respect to the  weight.
\end{rmk}

\begin{proposition}\label{convergence}
If $(\vs)_{s\in[0,1]}$ is a sequence of solutions of problems
\eqref{pqt.reg}, then the next properties hold:
\begin{itemize}
\item[(i)] There exists a constant $k$ dependent neither on $s$ nor on $l$ such that
$$
\nrm{\ddn{\vs}}{\ws{0}{2}{\ud}{\dppls}} \leq k \nrm{g}{\ws{1}{2}{\ud}{\dppl}}, 
$$
\item[(ii)] there exists $\varpi\in \ws{0}{2}{\ud}{\dppls}$  s.t. 
$$
\ddn{\vs}\rightharpoonup \varpi,\quad \text{ in } \ws{0}{2}{\ud}{\dppls},
$$
\item[(iii)] for every function $u\in \ws{1}{2}{0}{\ppl}$ one has at the limit $s\equiv0$, the  Green's formula:
$$
\int_{\ppl}^{}\nabla v\cdot \nabla u \,dy = \int_{\dppl}^{}\varpi u\,d\sigma (y),
$$
where $v$ is the solution of problem \eqref{qpt}.
\end{itemize}
\end{proposition}

\begin{proof}
Part (i) comes from Proposition \ref{rellich} combined with
Lemma \ref{approx}. Again we approximate $g\in \ws{1}{2}{\ud}{\dppl}$
by $g^\delta \in {\cal E}(\dppl)$ and we call $\vsd$ the unique solution of problem
\eqref{pqt.reg} with data $g^\delta$ given on $\ppls$. %Thanks to 
%Lemma \ref{approx} combined with arguments similar to the
%proof of Proposition 
By continuity of the solution of problem \eqref{pqt.reg} 
with respect to   the data, one easily shows that
$$
\begin{aligned}
\nrm{\nabla(\vsd-\vs)}{L^2(\ppls)}& \leq k \nrm{ R_s(g^\delta) - R_s(g)}{\ws{1}{2}{0}{\ppls}}  \leq k' \nrm{g^\delta -g}{\ws{1}{2}{\ud}{\dppls}} \\
& \leq k''\nrm{g^\delta -g}{\ws{1}{2}{\ud}{\dppl}}
\end{aligned}
$$
thanks to the weighted Rellich estimates above, one has  also that
$$
\nrm{ \ddn{\vsd} - \ddn{\vs}}{ \ws{0}{2}{\ud}{\dppl}} \leq k \nrm{g^\delta -g}{\ws{1}{2}{\ud}{\dppl}}.
$$
And because $g^\delta$ is regular, $\vsd\in H^2_\loc(\ppls)$. This allows us to write 
the  Green's formula for every $h \in \D(\ov{\ppls})$:
$$
\int_{\ppls}^{}\nabla \vsd \nabla h \,d y = \int_{\dppls}^{}\ddn{\vsd} h\,d\sigma(y)
$$
Thanks to strong convergence shown above, one can let
$\delta$ go to the limit and get:
\begin{equation}\label{Green.dom.reg}
\int_{\ppls}^{}\nabla \vs \nabla h \,d y = \int_{\dppls}^{}\ddn{\vs} h\,d\sigma(y)
\end{equation}
note that working only with $\ws{1}{2}{\ud}{\dppls}$
one can not write directly the Green's
formula: $\vs$ is not regular enough. 
Thanks to the last estimate of Lemma \ref{approx},
one has that 
$$
\lim_{s\to 0}\int_{\ppls}^{}\nabla \vs \nabla h \,d y = \int_{\ppl}^{}\nabla v \nabla h \,d y 
$$
Thus a limit for the boundary term in the r.h.s. of \eqref{Green.dom.reg} exists.
We express the boundary term for a fixed $s$ in $\pplst$ in $z$ 
coordinates, and we choose $h$ to be only a function of 
$z_1$ on the boundary. Moreover $\ti{\nu}:=(\sgn(z_1),-1)$ is the 
limit outward normal, (resp. $\ti{\nu}_s:=(a'(z_1),-1)$) and we write
$$
\begin{aligned}
  \lim_{s\to 0} \int_{\RR}^{}\nabla v(z_1,a_s(z_1))\cdot\ti{\nu} h(z_1)\,dz_1 = &\int_{\ppl}^{}\nabla v \nabla h \,d y \, \,+\\
 + \lim_{s\to0} & \int_{\RR}^{}\nabla v(z_1,a_s(z_1))\cdot \{ \ti{\nu} - \ti{\nu}_s \} h(z_1)\,dz_1 .
\end{aligned}
$$
The last term can be estimated through  the Cauchy-Schwartz
inequality
$$%\begin{aligned}
\int_{\RR}^{}\nabla v(z_1,a_s(z_1))\cdot \{ \ti{\nu} - \ti{\nu}_s \} h(z_1)\,dz_1 \leq \nrm{g}{\ws{1}{2}{\ud}{\dppl}}\nrm{(\sgn-a_s')h}{\ws{0}{2}{-\ud}{\dppl}}
%\end{aligned}
$$
Applying the Lebesgues theorem, the last term in the latter r.h.s. goes to zero. Indeed,
note that, by Poincar{\'e}-Wirtinger arguments, there exists a constant
$k$ s.t.
$$
\nrm{h}{\ws{0}{2}{-\ud}{\dppl}} \leq k \nrm{h}{\ws{1}{2}{0}{\ppl}}.
$$
By density and continuity arguments, one extends the Green's 
formula to all functions in $\ws{1}{2}{0}{\ppl}$.
\end{proof}

\begin{proposition}\label{extension}
  Let $T$ be the linear continuous operator from $\ws{1}{2}{\ud}{\dppl}$ on
$\ws{0}{2}{\ud}{\dppl}$ s.t. $T(g)=\ddn{v}$ where the  normal derivative
is to be understood as a weak limit exhibited above. Then $T$ is extended
as a map from $\ws{0}{2}{-\ud}{\dppl}$ on $\ws{-1}{2}{-\ud}{\dppl}$
where $\ws{ -1}{2}{-\ud}{\dppl}=(\ws{1}{2}{\ud}{\dppl})'$.
\end{proposition}

\begin{proof}
 $\ws{1}{2}{\ud}{\dppl}$ is dense in $\ws{0}{2}{-\ud}{\dppl}$. Let
$g,h$ be two functions of $\ws{1}{2}{\ud}{\dppl}$ such that 
$$
\left\{ 
\begin{aligned}
\Delta u = 0,\quad \text{ in } \ppl, \\
u=g ,\quad \text{ on } \dppl,  
\end{aligned}
\right. \quad 
\text{ and } \quad 
\left\{ 
\begin{aligned}
\Delta v = 0,\quad \text{ in } \ppl, \\
v=h ,\quad \text{ on } \dppl.
\end{aligned}
\right.
$$ 
Then the Green's formula from Proposition \ref{convergence} applies twice, giving
$$
\int_{\dppl}^{}u (\ddn{v}) \,d\sigma(y)= \int_{\dppl}^{} (\ddn{u}) v \,d\sigma(y).
$$
Thanks to the Rellich estimates, one then gets 
$$
\int_{\dppl}^{}(\ddn{u}) h \,d\sigma(y) \leq k \nrm{g}{\ws{0}{2}{-\ud}{\dppl}}\nrm{h}{\ws{1}{2}{\ud}{\dppl}},
$$
which gives that
$$
\nrm{T(g)}{\ws{-1}{2}{-\ud}{\dppl}}\leq k \nrm{g}{\ws{0}{2}{-\ud}{\dppl}},
$$
and density arguments complete the proof.
\end{proof}

\subsection{Weighted dual estimates on the normal derivatives}
We now return to the study of the vertical boundary layer
corrector that solves problem \eqref{QuaterPlane}. 
Thanks to Proposition \ref{extension}, we derive one of the key point estimates of the paper:

\begin{proposition}\label{prop.weak.estimates.micro}
There exists a unique solution $\tin \in \ws{1}{2}{0}{\Pi}$, of
problem \eqref{QuaterPlane}. Moreover there exists a constant $k$ that does not depend on $l$ s.t.
$$
\nrm{\ddn{\tin}}{\ws{-1}{2}{-\ud}{\dppl}} \leq k \left(\frac{1}{l}\right)^{1-\frac{1}{2M}},
$$
where the constant $M$ is defined as in Theorem \ref{BrBoMi}.
\end{proposition}

\begin{proof}
The proof is a straightforward application
of Proposition \ref{extension}:
$$
\nrm{\ddn{\tin}}{\ws{-1}{2}{-\ud}{\dppl}} \leq k \nrm{\tin}{\ws{0}{2}{-\ud}{\dppl}}. 
$$
As $\tin$ is at least $C^0$ inside the domain, we use the point-wise
 $L^\infty$ estimates from Theorem \ref{BrBoMi} which give:
$$
\nrm{\tin}{\ws{0}{2}{-\ud}{\dppl}}^2 \leq k'\int_{l}^{+\infty} \frac{1}{\rho^{3-\frac{1}{M}}} \,d\rho.
$$ 
That provides the desired result.
\end{proof}
\begin{rmk} This result express the decrease of the normal
derivative of $\tin$ on a vertical interface located at $y_1=l$.
These estimates improve the convergence
rate obtained in Proposition 4 in \cite{BrBoMi} by a factor 
of almost $\sqrt{1/l}$. Indeed we consider here the 
$\ws{-1}{2}{-\ud}{\dppl}$ norm, while in \cite{BrBoMi}, only the $\ws{-\ud}{2}{0}{\dppl}$ norm was used. In the rest of the article we imbed and exploit the result above into the macroscopic very weak setting.
\end{rmk}
%\section{Macro Sobolev versus micro weighted Sobolev norms}
\section{Correspondence between macro and micro Sobolev norms}
\label{micro-macro}
In \cite{BrBoMi}, a correspondence was  shown between $H^{\ud}_0(\gio)$ and a subspace of $\ws{\ud}{2}{0}{\dpl}$,
we extend it here between 
$H^{1}_0(\giop)$  and a subspace of $\ws{1}{2}{\ud}{\dppl}$ test functions. In what follows
the same could be written for $\Goutp$.
Taking $v\in H^1_0({\Gamma_{\rm in}}')$ we set 
$$
\tv\left(\frac{1}{\epsilon},\frac{x_2}{\epsilon}\right) = 
\tv\left(\frac{1}{\epsilon},y_2\right) = v(0,x_2),\quad \forall x_2 \in [0,1],
$$
and we extend  $\tv$ by zero on $\dppl$. Note that this makes sense
 because $v$ is  zero at $x_2=0$ and $x_2=1$, so that one has
\begin{lemma}\label{lemme.micro.macro}
  For a given function $v\in H^1_0({\Gamma_{\rm in}}')$ and $\tv$ defined above, 
the following equivalence of the  Sobolev trace norms occurs:
$$
\nrm{\tv}{\ws{1}{2}{\ud}{\dppl}} \leq k \nrm{v}{H^{1}_0({\Gamma_{\rm in}}')} \leq k' \nrm{\tv}{\ws{1}{2}{\ud}{\dppl}} 
$$
where the constants $k,k'$ do not depend on $\epsilon$.
\end{lemma}

\begin{proof}
  We start from the macroscopic side, the other way follows the same.
$$
\begin{aligned}
\int_{{\Gamma_{\rm in}}}  v^2(0,x_2) dx_2 & = \epsilon \int_0^{\ue} v^2(0,\epsilon y_2) dy_2 = \epsilon \int_0^{\ue} \tv^2\left(\ue,y_2\right) dy_2   \\
& \leq \epsilon \sup_{ y_2 \in [0,\ue]} \sqrt{1+y_2^2 + \left( \frac{1}{\epsilon} \right)^2 } 
\nrm{\tv}{\ws{1}{2}{\ud}{E'_{\frac{1}{\epsilon}}}}^2 \leq k \nrm{\tv}{\ws{1}{2}{\ud}{\dppl}}^2
\end{aligned}
$$
where the constant $k$ is obviously independent on $\epsilon$. Owing that $\partial_{y_2}\tv=\epsilon \partial_{x_2} v$, the derivative part is shown similarly.
\end{proof}

\section{Very weak estimates for boundary layer and wall law approximations}
\label{section.vws}
Turning again to the macroscopic error estimates, one defines the error
$\rui:= \uenp-\uiue$ where $\uenp$ is the exact solution of problem
\eqref{RugueuxCompletNP} and $\uiue$ the boundary layer approximation proposed in  \eqref{fbla}.
It satisfies the  set of equations:
\begin{equation}\label{eq.err} 
\left\{ 
\begin{aligned} 
&\Delta \rui = 0,\quad \text{ in } \Oe \\
& \rui = 0,\quad  \text{ on  } \Geps \\
& \rui = - \epsilon  \dd{u^1 }{x_2}(x_1,0)\left( \left( \beta - \obeta +\tin \right) \left(\frac{x_1}{\epsilon},\frac{1}{\epsilon} \right)+  \tout\left(\frac{x_1-1}{\epsilon},\frac{1}{\epsilon}\right)\right),\text{ on } \Gun, \\
& \ddn{\rui} =   -\ddn{\tin}\left(\frac{1}{\epsilon},\frac{x_2}{\epsilon}\right) \text{ on } \Gout ,\quad \ddn{\rui} =   - \ddn{\tout}\left(\frac{1}{\epsilon},\frac{x_2}{\epsilon}\right) \text{ on } {\Gamma_{\rm in}}. \\%\quad 
\end{aligned}  
\right. 
\end{equation}
In order to improve $L^2(\Oz)$ estimates obtained in \cite{BrBoMi}, we use the
material  above to prove the main result of this paper:

\begin{thm}\label{main}
  There exists a unique solution $\rui \in H^1(\Oe)$ of problem \eqref{eq.err};
it satisfies the estimate:
$$
\nrm{\rui}{L^2(\Oz)}  \leq k \epsilon^{\min\left(\td+\alpha, 2 - \frac{1}{2M}\right)},
$$
 the constants $\alpha$ and $M$ being defined in Theorem \ref{BrBoMi}.
\end{thm}

\begin{proof}
For any given function $F\in L^2(\Oz)$, we solve the  regular problem:
find $v\in H^1_D(\Oz):=\{ u \in H^1(\Oz) \text{ s.t. } u=0  \text{ on } \gzu\}$
such that
$$
\left\{ 
\begin{aligned}
&-\Delta v = F ,\text{ in } \Oz, \\
& \ddn{v}=0, \text{ on } \giop, \\
& v=0,\text{ on } \Gz\cup \Gun.
\end{aligned}
\right.
$$
  
According to Theorem 4.3.1.4, p. 198 \cite{Grisvard}, 
$v\in H^2(\Oz)\cap H^1_D(\Oz)$ so that $v \in H^1(\partial \Oz)$ and, thanks to boundary conditions
on $\gzu$,  $v \in H^1_0(\gio)$. We are now in the position to apply 
the Chapter 5 of \cite{Ne.Book.67} to write that:
$$
\int_{\Oz}^{}\rui F \,dx = -\left( \rui, \ddn{v} \right)_{\gzu} + \left< \ddn{\rui}, v \right>_{\gio},
$$
where by the brackets we denote the duality pairing 
$H^{-1},H^1_0(\gio)$ and by the parentheses we denote the scalar product
in $L^2(\gzu)$. By standard interior regularity results
one  easily gets that $\tin\in H^2_{\loc}(\Pi)$ (resp. $\tout \in H^2_{\loc}(\Pi)$) so that the normal
derivatives 
$$
\ddn{\tin} \left( \ue, \frac{\cdot}{\epsilon} \right) \in L^2(0,1), \left(\text{resp. } \ddn{\tout} \left( \ue, \frac{\cdot}{\epsilon} \right)  \in L^2(0,1) \right) .
$$
for every fixed $\epsilon$. Thus the duality pairing becomes
an integral:
$$
\begin{aligned}%
\left< \ddn{\rui}, v \right> & = - \int_{\Gout}\ddn{\tin}  \left( \ue, \frac{x_2}{\epsilon} \right) v(x) d\sigma (x) - \int_{{\Gamma_{\rm in}}}\ddn{\tout}  \left( \ue, \frac{x_2}{\epsilon} \right) v(x) d\sigma (x) \\
& = - \epsilon \int_{0}^{\ue}\left\{ \ddn{\tin}\left( \ue, y_2 \right)\tvi\left( \ue, y_2 \right)+ \ddn{\tin}\left( \ue, y_2 \right)\tvo\left( \ue, y_2 \right) \right\}  \,dy_2 \\
&\leq \epsilon \left( \nrm{\ddn{\tin}}{\ws{-1}{2}{-\ud}{\partial \Pi_{\ue}'}}  \nrm{\tvo}{\ws{1}{2}{\ud}{\partial \Pi_{\ue}'}} 
+ \nrm{\ddn{\tout}}{\ws{-1}{2}{-\ud}{\partial \Pi_{\ue}'}}  \nrm{\tvi}{\ws{1}{2}{\ud}{\partial \Pi_{\ue}'}} \right)\\
& \leq  \epsilon \left( \nrm{\ddn{\tin}}{\ws{-1}{2}{-\ud}{\partial \Pi_{\ue}'}} + \nrm{\ddn{\tout}}{\ws{-1}{2}{-\ud}{\partial \Pi_{\ue}'}} \right) \nrm{v}{H^1_0(\giop)}
\end{aligned}
$$
where $\tvi$ and $\tvo$ are the microscopic test functions associated to
the trace of $v$ on $\giop$ as in section \ref{micro-macro}. 
One then concludes this part setting $l=1/ \epsilon$ in Proposition \ref{prop.weak.estimates.micro}.
The $L^2(\gzu)$ scalar product has  been estimated in \cite{BrBoMi}, 
using {\em a priori} estimates for the $\Gz$ part whereas the $\Gun$
part uses again $L^\infty$ estimates from Theorem \ref{BrBoMi}.
\end{proof}
A direct consequence of this result is 
\begin{thm}\label{thm.wall.law}
  The first order wall law solving
\begin{equation} \label{fowl}
\left\{ 
\begin{aligned} 
&\Delta \uu = 0 , \quad \text{ in } \Oz, \\
& \uu = \ov{U}  , \quad \text{ on } \Gun, \\
& \uu = \epsilon \obeta \dd{\uu}{x_2}  , \quad \text{ on } \Gz, \\
& \ddn{ \uu}= 0 , \quad \text{ on } \gio, \\
\end{aligned}  
\right. 
\end{equation}
satisfies the  error estimate 
$$
\nrm{\uenp-\uu}{L^2(\Oz)} \leq k \epsilon^\td,\quad 
$$
where the constant $k$ is independent on $\epsilon$.
\end{thm}
The proof follows exactly the same line as in Theorem 5.3 in \cite{BrBoMi},
but the result is improved thanks to the  Theorem \ref{main} above.

\section{Numerical evidence}
\label{num}
We define the rough bottom of the domain by setting $f$ in \eqref{domain} as:
$$
f(y_1)= -1+\nud\sin(2\pi y_1), \quad \forall y_1 \in [0,1].
$$
This is obviously a Lipschitz  smooth function 
compatible with the hypotheses of the claims. 
In what follows we look for a numerical validation of theoretical 
convergence results above: we compute 
for every fixed $\epsilon\in[0,1]$
\begin{itemize}
\item[-] $\ueh$ a numerical approximation of $u^\epsilon$ solving a discrete counterpart of problem \eqref{RugueuxCompletNP}.
\item[-] $\uiueph$,  the periodic full boundary layer approximation  (it does not contain any vertical corrector) defined in \eqref{uiuep}
\item[-] $\uiueh$, the full boundary layer approximation including vertical correctors defined in \eqref{fbla}
\item[-] $u^1_h$, the averaged wall-law presented in \eqref{fowl}, and $u^0$ the zero order approximation.
\end{itemize}
 
We use the finite element method code {\tt freefem++} \cite{ffm}, in order
to compute $\ueh, \beta_h, \tinh$ and $\touth$. The $\PP_2$ Lagrange finite elements interpolation is chosen.

\paragraph{Microscopic correctors}

As $\beta,\tin$ and $\tout$ are defined on infinite domains, we have to truncate
these and set up proper boundary conditions on the corresponding new boundaries. For $\beta$,
this was  analysed in \cite{JaMiNe.01} so that we only need
to solve
$$
\left\{
\begin{aligned}
& -\Delta \beta_L = 0,\text{ in} \quad \zup \cap \{y\in \rr^2 y_2<L\} ,\\
& \beta_L = - y_2,\text{ on } P^0,\\
& \beta_L \text{ is } \yup,\\
& \ddn{\beta_L}=0, \text{ on } \{ y_2=L\}.
\end{aligned}
\right.
$$
The approximation $\beta_L$ is exponentially close
to $\beta$ with respect to  $L$ in the Dirichlet norm (see Proposition 4.2 \cite{JaMiNe.01}).
For the vertical correctors we set the domain
$\pil:=\Pi\cap [-1,L]^2$ (resp. $\pilm:=\Pi_-\cap [-L,1]^2$
and we solve the  problem
\begin{equation}\label{approx.xin}
\left\{
\begin{aligned}
& -\Delta \tinl = 0, \quad \text{ in } \Pi, \\
& \ddn{\tinl}(0,y_2) = - \ddn{ \beta }(0,y_2) ,\quad \text{ on } E,\\
& \tinl = 0, \quad  \text{ on } B,\\
& \ddn{\tinl}=0, \text{ on } \{ y_1=L \} \cup\{ y_2=L \}=:G_L ,
\end{aligned}
\right.
\end{equation}
the symmetric problem for $\toutl$ being omitted.
By Proposition 4 in \cite{BrBoMi} and Proposition \ref{prop.weak.estimates.micro} above,
one easily deduces the  convergence result:

\begin{proposition}\label{prop.approx.xin}
  There exists a unique solution $\tinl \in \ws{1}{2}{0}{\pil}$
solution of problem \eqref{approx.xin}, moreover one has
$$
\begin{aligned}
\nrm{\tinl-\tin}{\ws{1}{2}{0}{\pil}}\leq k L^{-\alpha},\quad 
\nrm{\tinl-\tin}{\ws{0}{2}{-1}{\pilp}}\leq k L^{-1+\frac{1}{2M}}
\end{aligned}
$$
where the constants $k,k'$ are independent of $L$ 
and $\alpha$ and $M$ are defined as in Theorem \ref{BrBoMi}.
$\pilp$ is the restriction of $\pil$ to $\R_+\times\R_+$.
\end{proposition}
In figures \ref{mesh.beta} and \ref{mesh.xi}, we display the meshes obtained after
 adaptative procedure, described below, for $\beta^L$ and $\toutl$. The total number
of vertices used in the meshes for discretising $\beta^L_h$, $\tinlh$ and $\toutlh$
are 39000, 78000 and 79000. In the simulation of $\beta^L_h$ the horizontal top is set to $L:=10$. 
For $\tinlh$ and $\toutlh$ the vertical interface is set to $L:=20$.
The contours of corresponding solutions $\beta^L_h$ and $\toutlh$ are displayed in figures \ref{beta}, \ref{xi.out} and \ref{zoom.xiout},
whereas the normal derivative $\ddn{\beta^L_h}$ and $-\ddn{\xi_{{\rm in},h}^L}$ 
are shown to coincide along $\{0\} \times [-\frac{3}{4},1]$  in figure \ref{dxbeta}.
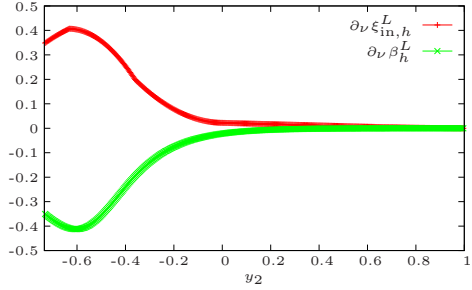
\begin{figure}[!h] 
\begin{center}
\input{dbdnE}

\end{center}
\caption{Normal derivatives $\ddn{\beta^L}$ and $\ddn{\toutl}$ on $E$, the vertical interface}
\label{dxbeta}
\end{figure}

We perform a single microscopic computation. Then 
we re-scale the boundary layer to the macroscopic 
domain setting 
$$
\beta^L_{\epsilon,h}(x):=\beta^L_{h}\lrxe, \quad \tinleh(x) := \tinlh \lrxe, \quad \toutleh(x) := \toutlh \lrxe, \quad \forall x \in \Oe.
$$
We quantify the interpolation error with respect to $\epsilon$. 
$$
\nrm{(\beta_{\epsilon,h}^L- \beta )\lrde }{L^2(\Oz)} \leq \sqrt{\epsilon} \left( \nrm{\beta_h^L- \beta^L}{L^2(\zp)} + \nrm{\beta^L- \beta}{L^2(\zp)} \right)\leq k \sqrt{\epsilon} h_m^s \snrm{\beta^L}{H^s(\zup)}
$$ 
where $s$ is a constant dependent on the boundary's regularity, and $h_m$ a fixed maximum
mesh size on the microscopic level, independent on $\epsilon$.
In the same way one can set
$$
\begin{aligned}
\nrm{(\toutlh- \tout)\lrde }{L^2(\Oz)} & \leq \epsilon  \left( \nrm{\toutlh - \toutl}{L^2(\pilm)} + \nrm{\toutl - \tout}{L^2(\pilm)} \right)\\
& \leq k \epsilon \left( h_m^2  \nrm{\toutl}{H^{2,\nu}(\Pi_-^L)}  + k L \nrm{\toutl -\tout}{\ws{0}{2}{-1}{\pimlp}} \right) \leq k \,\epsilon ,
\end{aligned}
$$
where $\nu$ is a real parameter depending on the angle of the corner of $\Pi_-$ at $(0,f(0))$,
and $H^{2,\nu}$ the  weighted space defined p.388 Definition 8.4.1.1 \cite{Grisvard}, that 
takes into account the corner singularity of second derivatives of $\toutl$. 
These estimates give an upper bound on the convergence rate for the full boundary layer $\uiue$, namely:
\begin{equation}\label{esti.num}
\begin{aligned}
\nrm{u^\epsilon_h - \uiueh}{L^2(\Oz)} & \leq \nrm{u^\epsilon_h - u^\epsilon }{L^2(\Oz)} + \nrm{ u^\epsilon - \uiue }{L^2(\Oz)} + \nrm{\uiue - \uiueh }{L^2(\Oz)} \\
&  \leq H^2 \nrm{u^\epsilon}{H^{2,\nu}(\Oe)}+ k \epsilon^\td ,
\end{aligned} 
\end{equation}
where $H$ is a macroscopic mesh size presented in the next paragraph.
%%la m{\^e}me borne supp{\'e}rieure s'{\'e}crit pour $\uiuep$.

\paragraph{Rough solutions}
When computing numerical approximations of $u^\epsilon$, one
has to play with 3 concepts that are interdependent:
$h$ the mesh-size, $\epsilon$ the roughness size, and corner
singularities that depend on the shape of the domain.

In the periodic case considered in \cite{BrMiQam}, 
and for $f \in C^\infty(]0,1])$, in order to avoid that the 
roughness size goes under the mesh-size,
one could discretise the solution on a 
mesh such that $h\leq c\, \epsilon$.
  Due to estimates on
the interpolation error and $H^2(\Oe)$ regularity,
one obtains a good numerical agreement for convergence
rates between theoretical and numerical results (see \cite{BrMiQam}).

In the non-periodic setting, corner singularities occur 
near $\Gout$. In order to obtain convergent numerical
approximations of $u^\epsilon$ near $\Gout$, one should 
refine the mesh in the neighbourhood of $(1,\epsilon f(1/\epsilon))$.
At the same time, in the regular zones, the mesh-size
should stil be refined at least linearly with respect to $\epsilon$
(as in the peridic setting \cite{BrMiQam}).
 This complicates the local size 
of elements with respect to the size of the mesh (\cite{Grisvard} p.384).
Thus, simply setting uniformly $h:=c \epsilon$ does not provide accurate convergence results.
On the other hand, one aims to have a strong
control on the mesh size far from the corner: for instance in these zones, 
 the mesh-size could be fixed on a uniform grid. 
These considerations led us to use an overlapping
Schwartz algorithm \cite{QuVaBook.99}; we split $\Oe$ in two parts:
$\Oz$ is discretised with a structured grid of size $H:= k\epsilon^\gamma$ ($\gamma$ is discussed later),
whereas a second domain reads
$$
\Ozu:= \Oe \cap \left\{ x \in \RR \; \text{ s.t. } \; x_2 < \frac{\epsilon}{10} \right\} 
$$
and contains the rough sub-layer. On $\Ozu$ we perform 
mesh adaptation in order to capture geometrical and corner singularities.
The maximum/minimum mesh-sizes are set:
$$
h_{\min} = \min_{K\in \Tk}  h_K, \quad h_{\max} = \min_{K\in \Tk} h_K
$$
where $h_K$ is the diameter of triangle $K$ in the triangulation $\Tk$ of $\Ozu$.
At each step $m$ of the Schwartz algorithm, we solve
two problems. We set $\cUm$ to be the solution of 
$$
\left\{ 
\begin{aligned} 
& \Delta  \cUm =0 , \text{ in } \Oz, \\
& \cUm = 1,  \text{ on } \Gun, \\
& \cUm = \cV^{(m-1)},  \text{ on } \Gz, \\
& \ddn{\cUm}=0, \text{ on } \giop, 
\end{aligned}  
\right. 
$$
and $\cVm$ solves 
$$
\left\{ 
\begin{aligned} 
& \Delta  \cVm =0 , \text{ in } \Ozu, \\
& \cVm = 0,  \text{ on } \Geps, \\
& \cVm = \cUm,  \text{ on } \left\{ x_2=\frac{\epsilon}{10}\right\} , \\
& \ddn{\cVm}=0, \text{ on } ( \gio ) \cap\left\{ x \in \RR \; \text{ s.t. } \; x_2 < \frac{\epsilon}{10} \right\}  , 
\end{aligned}  
\right. 
$$
and we iterate the procedure until 
$$
\int_{(0,1)\times \{ 0 \}\cup(0,1)\times \{ \frac{\epsilon}{10} \}  } (\cUm-\cVm)^2 d\sigma(x) < {\rm tol},
$$ 
where tol is a constant set to 10$^{-10}$.
During this step both meshes are kept fixed.

Then we refine the sub-layer mesh $\Tk$ in order to account
the corner singularity. This step provides a new mesh-size
distribution updating $\hm$ and $\hM$. We use adaptative
techniques presented p. 92 of the {\tt freefem++} reference manual \cite{ffm}.
This  procedure is 
compatible with the mesh requirements displayed in Theorem 8.4.1.6 p. 392 in \cite{Grisvard} and guarantees  standard interpolation errors
with respect to the mesh size.

We iterate these two steps: solve the Schwartz domain decomposition 
problem and then adapt the mesh. The iterative algorithm stops
when $\hM < H$. Through this algorithm we insure both a 
given mesh size $H$ and a refined mesh near the corner.

We tested different values of $\gamma$ where setting $H = k \epsilon^\gamma$, $k$  a is  given constant, 
choosing $\gamma\geq \frac{5}{4}$ does no more change convergence  results below.
We plot in fig. \ref{mshsize}, $\hM$ and $\hm$ as functions
of $\epsilon$. The adaptative process gives approximately $\hm \sim c \epsilon^{2.29}$.

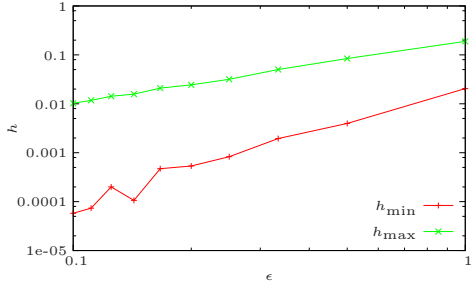
\begin{figure}[!h]
\begin{center}
\input{mshsize}
\end{center}
\caption{Mesh sizes $\hm$ and $\hM$ as functions of $\epsilon$}\label{mshsize}
\end{figure}
We plot in fig. \ref{fig.rough.meshs}, the meshes obtained thanks to
our iterative scheme for $\epsilon\in\{ \ud,\frac{1}{3} \}$. In fig. \ref{fig.rough.sols}, we display the corresponding solutions $u^\epsilon_h$. Next,
we construct boundary layers 
using microscopic correctors above.
We compute the errors $\ueh - u^0$, $\ueh - u^1$, 
$\ueh-\uiueph$ and $\ueh-\uiueh$ in the 
$L^2(\Oz)$ norms, and display them as a function of $\epsilon$ in fig.\ref{err}.
\begin{figure}[!h]
\begin{center}
\input{errl2_dd_coarse} 
\input{errh1_dd_coarse}
\end{center}
\caption{Errors in the $L^2(\Oz_h)$ (left) and the $H^1(\Oe_h)$ (right) norms with respect to  $\epsilon$}\label{err} 
\end{figure}
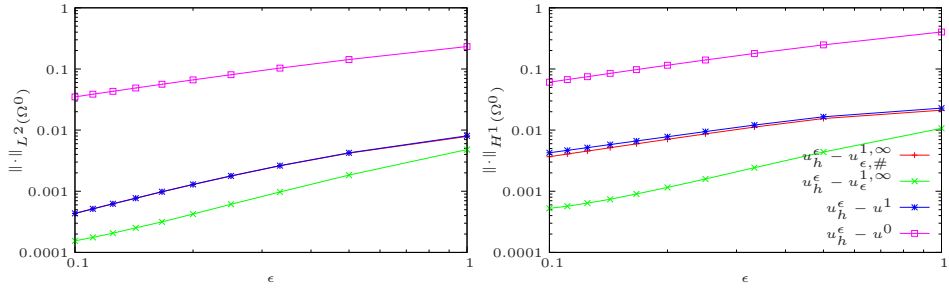
The numerical convergence rate, obtained by interpolating results above as a powers of $\epsilon$, is displayed in table \ref{table}.
\begin{table}[h!]
\begin{center}
\begin{tabular}{|c|c|c|c|c|}
\hline
norm / approx.& $\ueh - u^0$ & $\ueh - u^1$ & $\ueh-\uiueph$ &  $\ueh-\uiueh$ \\
\hline
\hline
$L^2(\Oz)$ & 0.78783 & 1.11 & 1.1&	1.462 \\
\hline
$H^1(\Oz)$ & 0.787 & 0.6869&  0.70 & 1.346347 \\
\hline
\end{tabular}
\end{center}
\caption{The errors convergence rates displayed as powers of $\epsilon$ }\label{table}
\end{table}

\paragraph{Discussion}
When the vertical correctors are not present, the boundary layer
approximation is not only less accurate but also the rate of convergence
is less than first order, the difference is visible in $L^2(\Oz)$
but is significant in the $H^1(\Oz)$ norm.
Nevertheless, and as explained above, when using a single microscopic computation
of the correctors for every $\epsilon$, it is not possible to 
get better convergence results than $\epsilon^\td$. This is actually 
what we obtain for our more accurate approximation $\uiueh$.
This validates our theoretical results. The surprising
phenomenon that we are at this point not able to 
justify is the poor convergence rate of the wall law $u^1$,
that should according to our estimates be $\epsilon^\td$. Observed
in  \cite{BrMiQam}, $u^1$ performs even worse convergence rate
than $u^0$ in the $H^1(\Oz)$ norm. The results of Theorem \ref{a.priori}
are fairly approximated for what concerns the $H^1(\Oz)$ error of $\uiueh$.

%%%%-------------------------------------------------------------------%%%%%

\section{Conclusion}
Our approach provides an almost complete understanding of the non-periodic
case for lateral homogeneous Neumann boundary conditions in the
straight case, (no curvature effects of the rough boundary \cite{NeNeMi.06}).
A forthcoming paper should adapt these results to the case
mentioned in the introduction: a smooth boundary forward and
backward the rough domain {\em via} domain decomposition techniques.
Another extension  to the Stokes system should follow as well.

\bigskip
\bigskip
\noindent {\it Acknowledgements.}
%\section{Acknoledgements}

The author would like to thank C. {\sc Amrouche} for his advises and support, 
S. {\sc Nazarov} for fruitful discussions and clarifications, and H. {\sc Teismann} for proof reading.
This research was partially funded by Cardiatis\footnote{www.cardiatis.com}, 
an industrial partner designing and commercializing metallic wired stents.

\bibliographystyle{plain}
\bibliography{eqred}

\newpage

\begin{figure}[!h] 
\begin{center}
\includegraphics[width=0.35\textwidth]{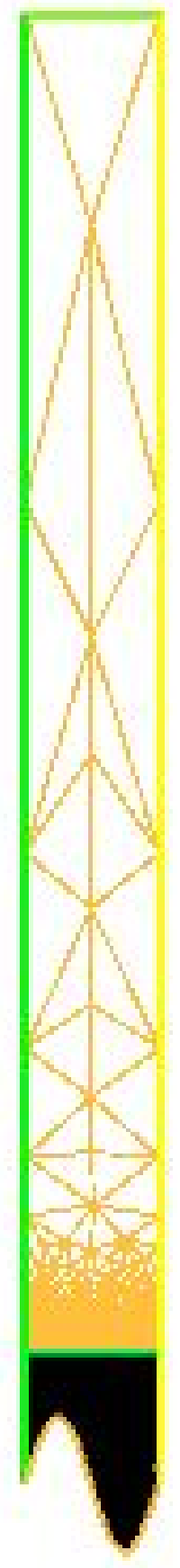}
\includegraphics[width=0.15\textwidth]{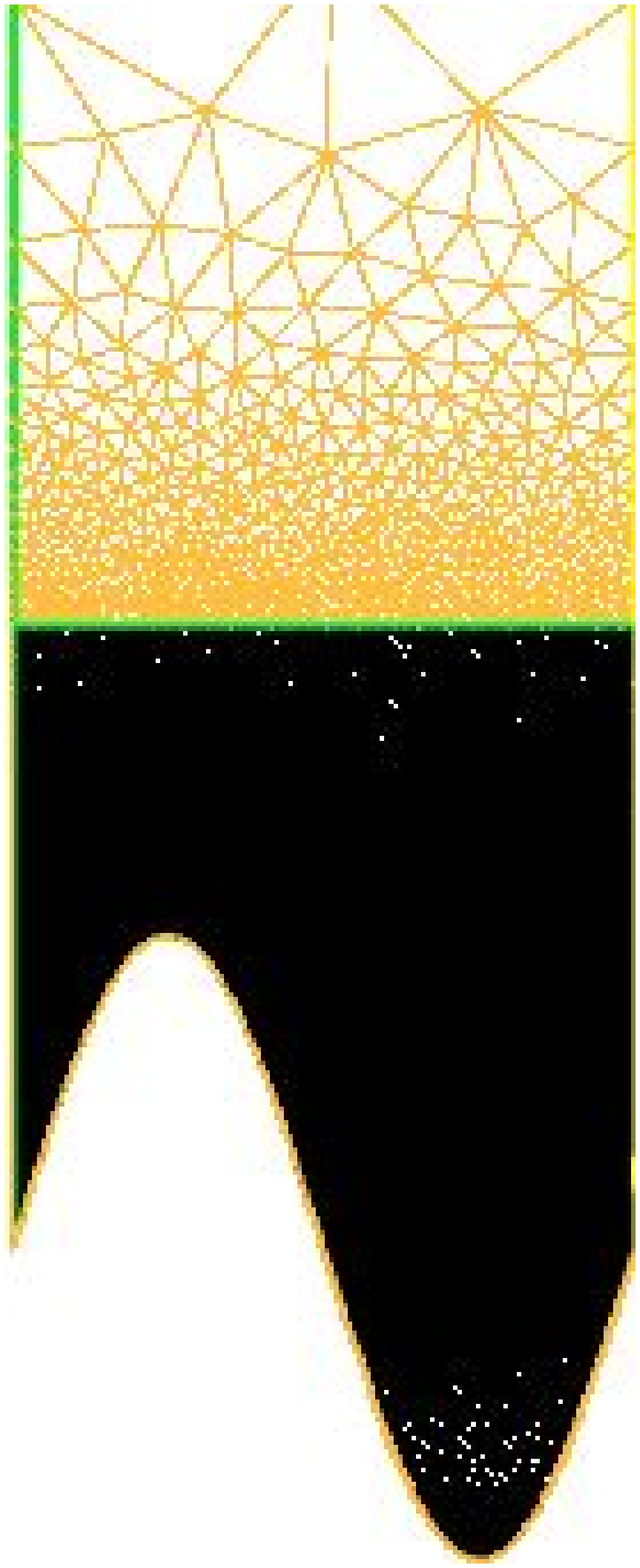}
\end{center}
\caption{The microscopic periodic cell after adaptative mesh refinement}
\label{mesh.beta}
\end{figure}

\begin{figure}[h!] 
\begin{center}
\includegraphics[width=0.45\textwidth]{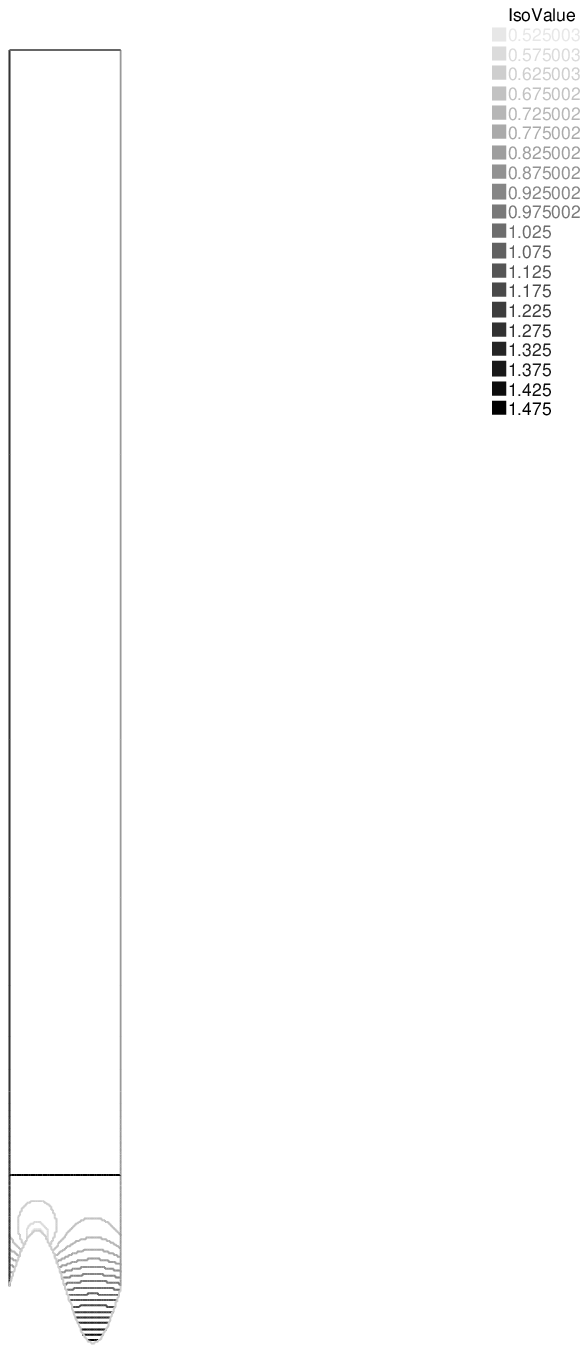}
\includegraphics[width=0.45\textwidth]{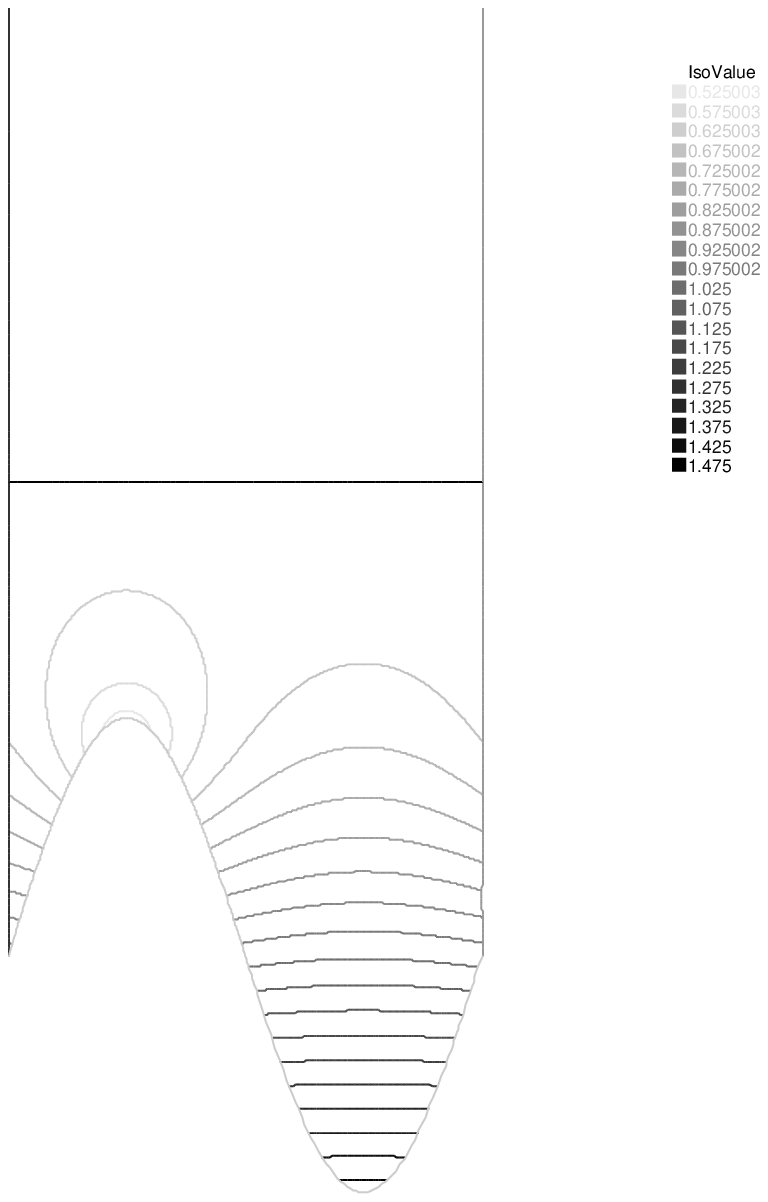}
\end{center}
\caption{The microscopic periodic cell corrector $\beta^L$}\label{beta}
\end{figure}
\newpage

\begin{figure}[h!] 
\begin{center}
\includegraphics[width=0.9\textwidth]{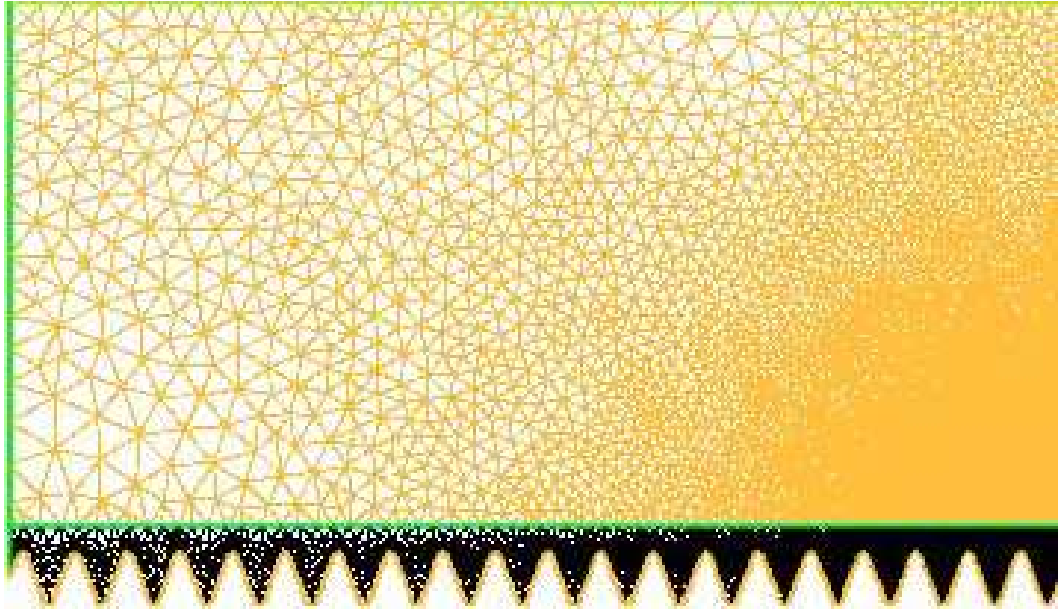}
\end{center}
\caption{The microscopic domain of $\toutl$  after adaptative mesh refinement}
\label{mesh.xi}
\end{figure}
\begin{figure}[h!]  
\begin{center}
\includegraphics[width=0.9\textwidth]{xiout.epsi}
\end{center}
\caption{The microscopic corrector $\toutl$}
\label{xi.out}
\end{figure}

\begin{figure}[h] 
\begin{center}
\includegraphics[width=0.79\textwidth]{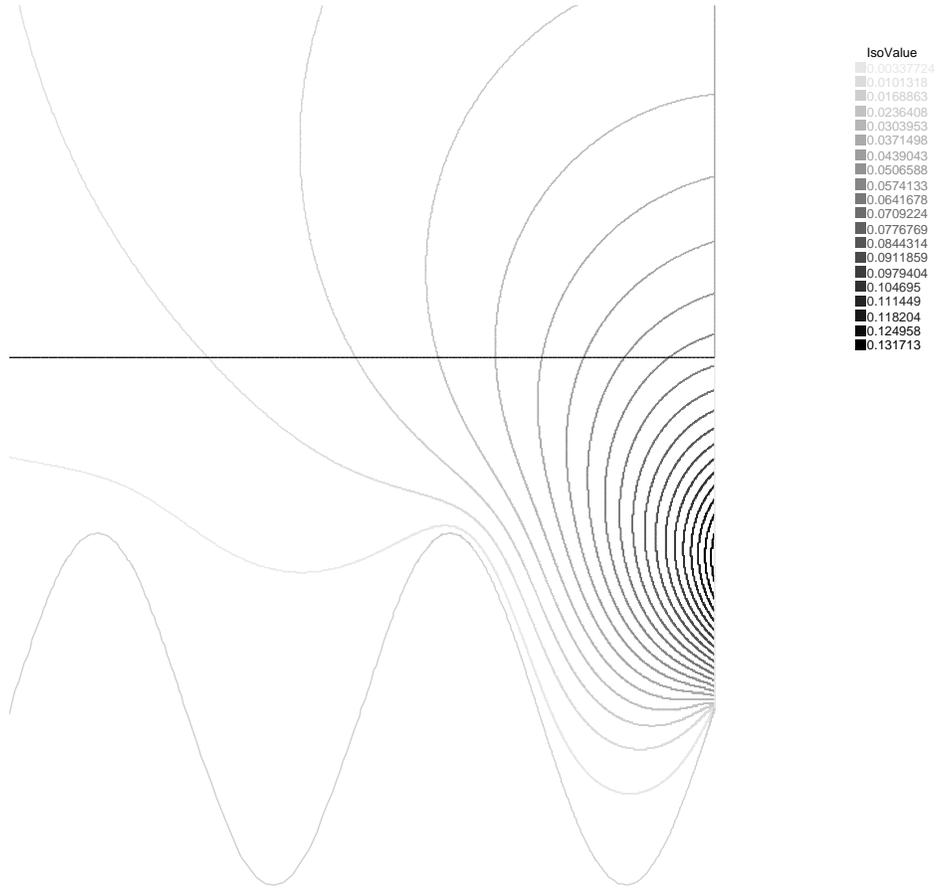}
\end{center}
\caption{A zoom near the corner singularity of the microscopic corrector $\touth$ }\label{zoom.xiout}
\end{figure}

\newpage

\begin{figure}[!h]
\begin{center}
\includegraphics[width=0.45\textwidth]{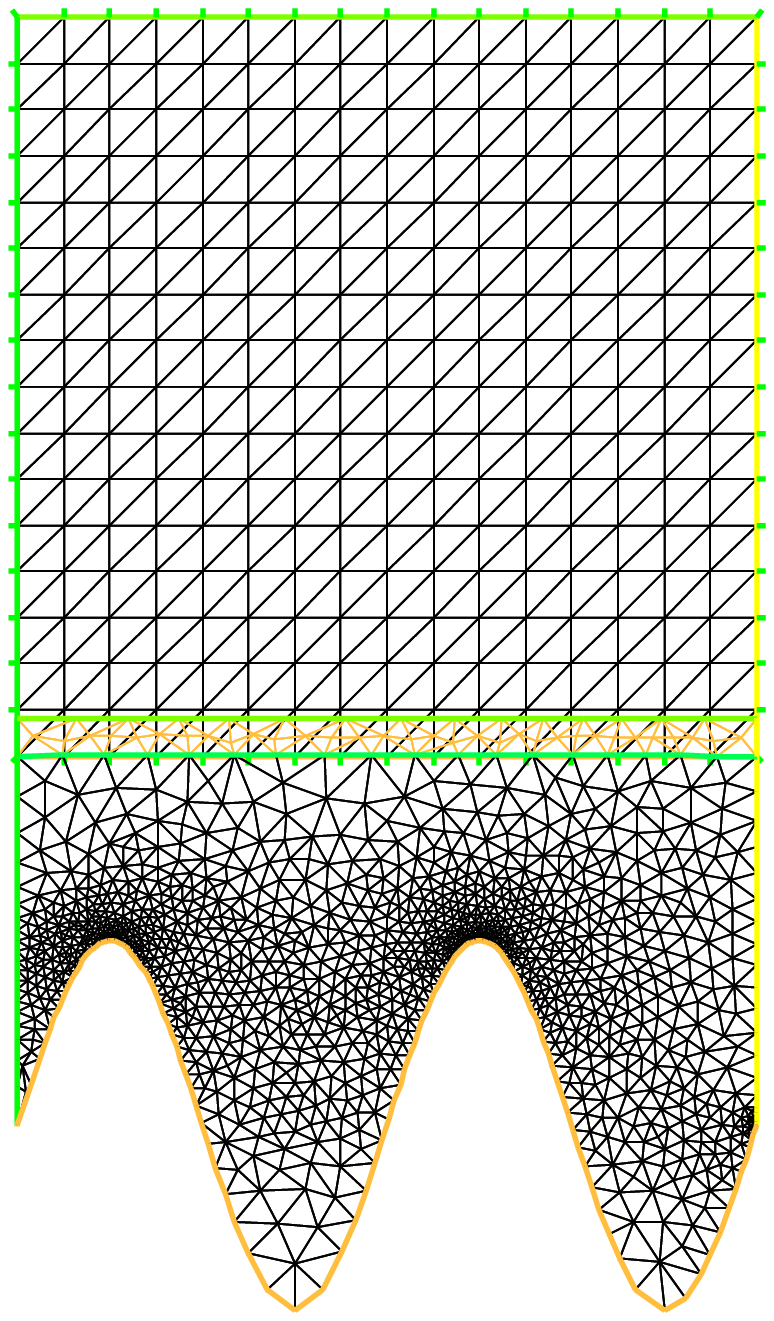}
\includegraphics[width=0.45\textwidth]{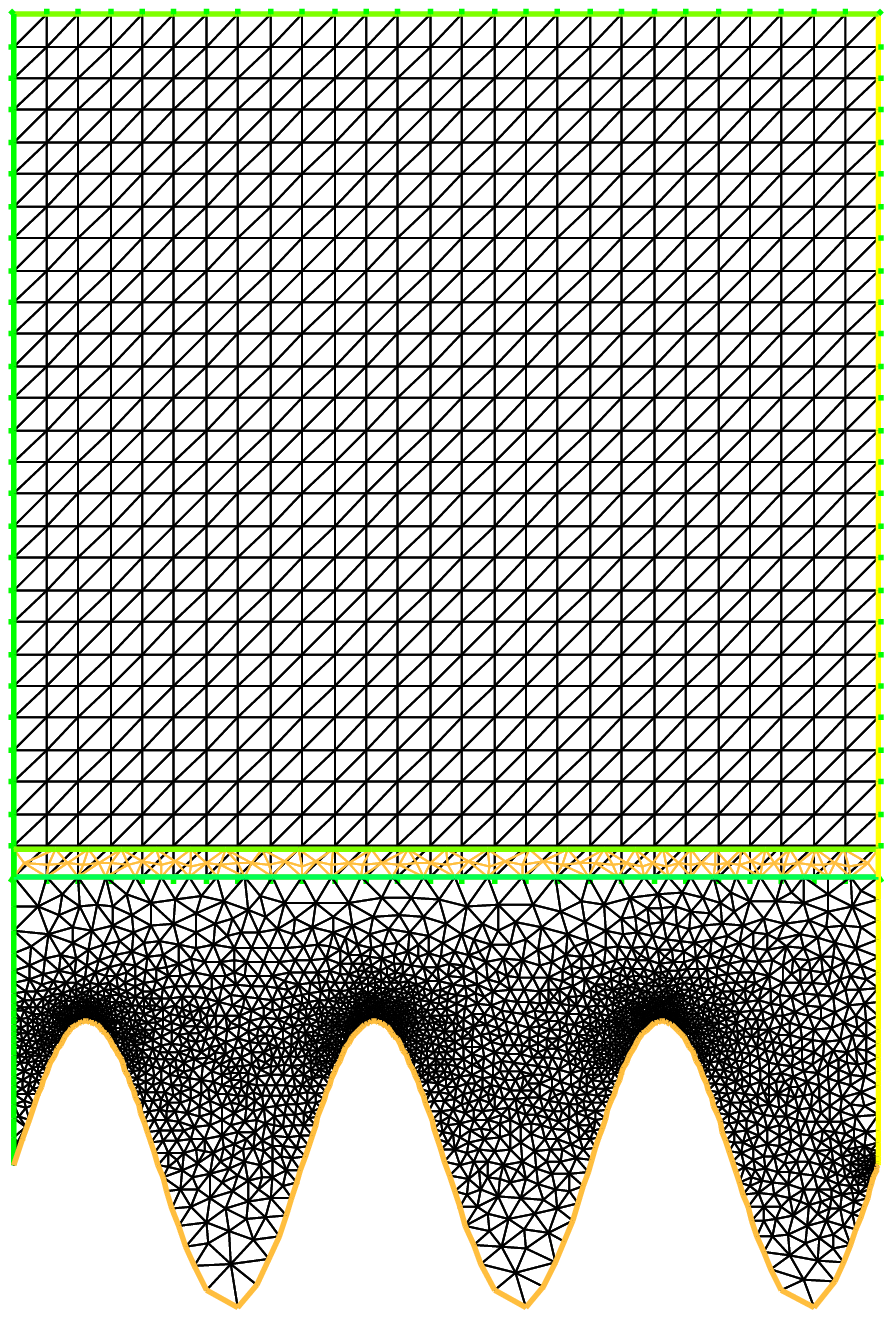}
\end{center}
\caption{The meshes for the rough solution for  $\epsilon\in\{\ud,\sfrac{1}{3}  \}$ using a decomposition method}\label{fig.rough.meshs} 
\end{figure}
\begin{figure}[!h]
\begin{center}
\includegraphics[width=0.45\textwidth]{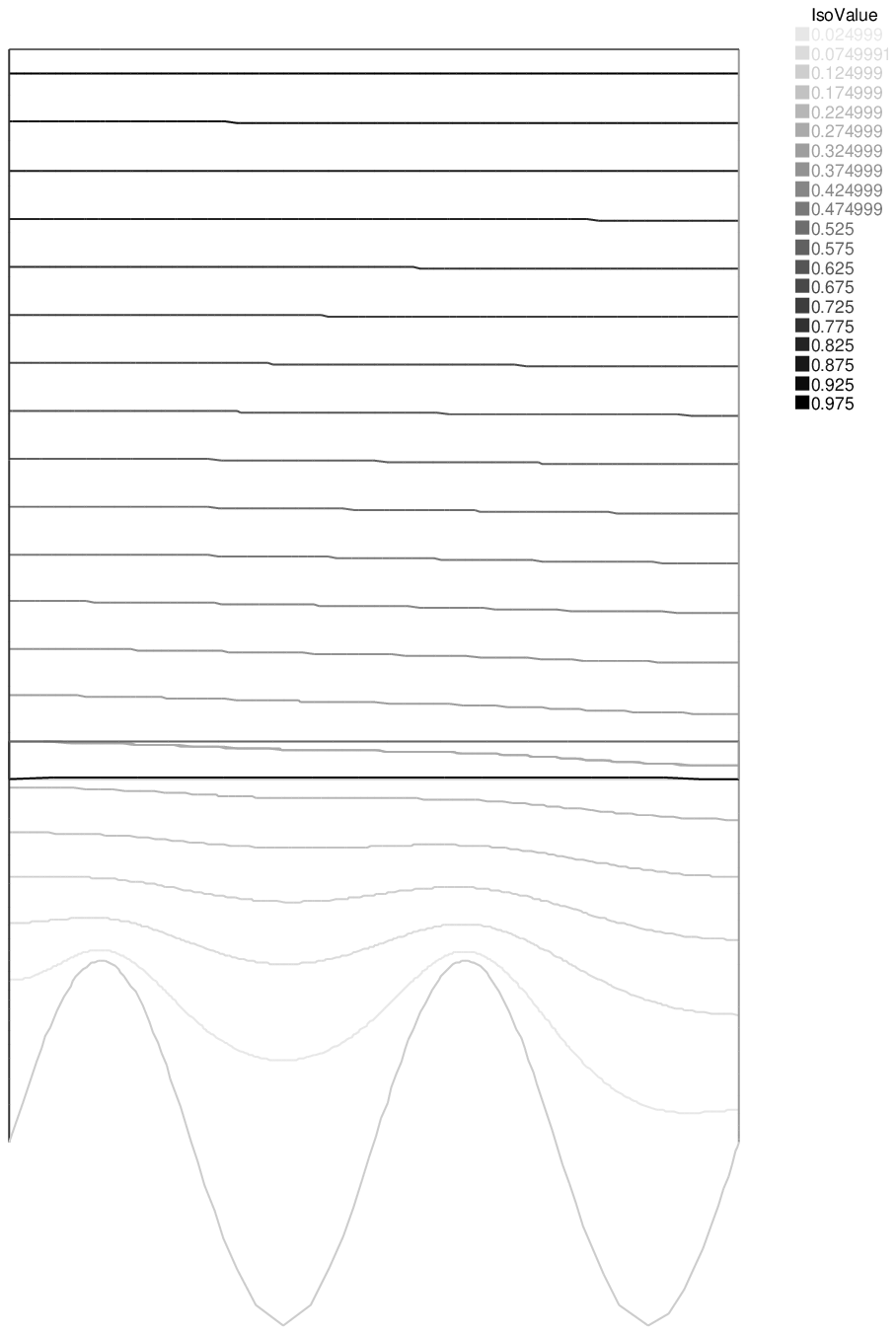}
\includegraphics[width=0.45\textwidth]{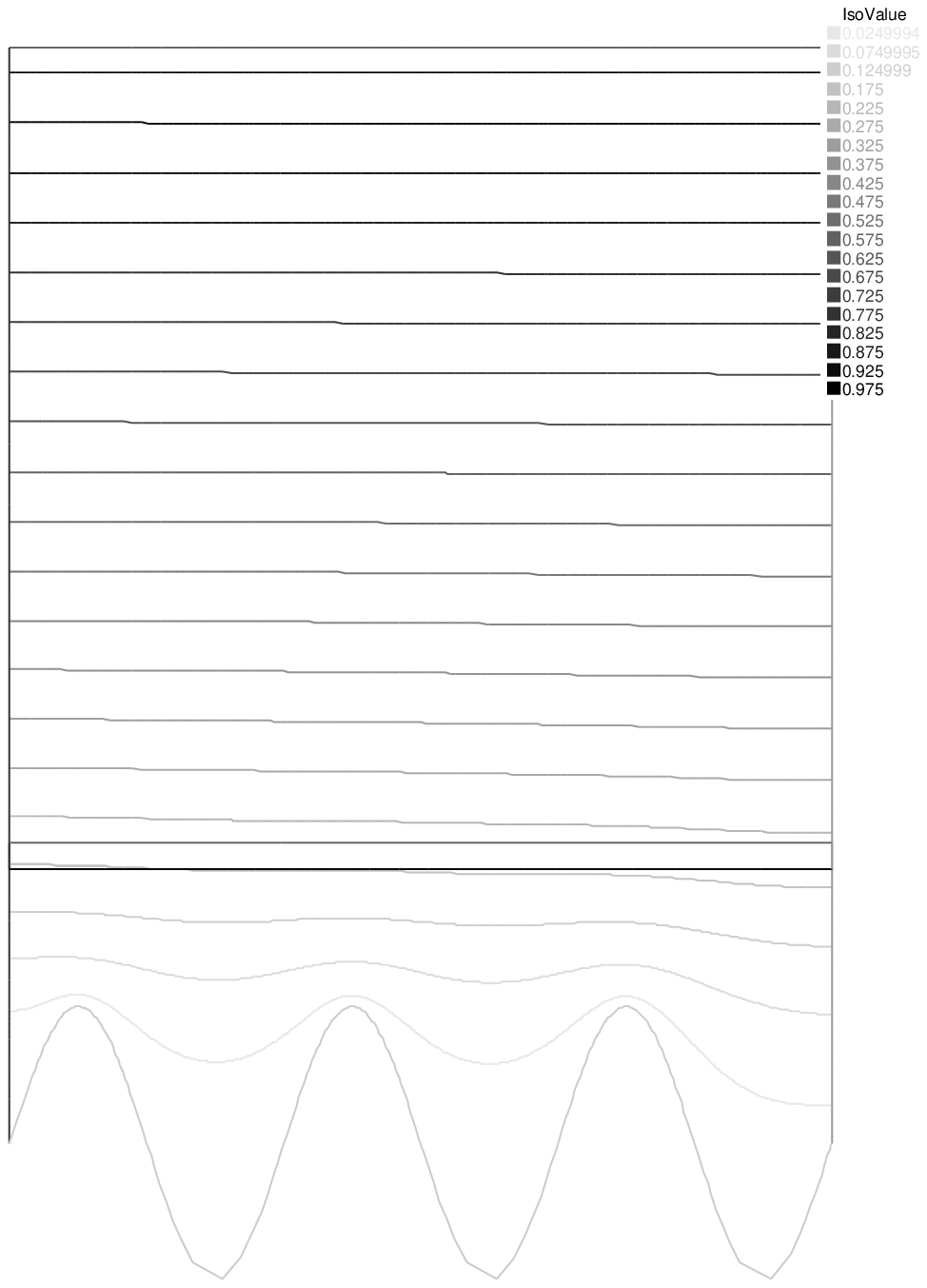}
\end{center}
\caption{The rough solution computed for  $\epsilon\in\{\ud,\sfrac{1}{3}  \}$ using a decomposition method}\label{fig.rough.sols} 
\end{figure}

\end{document}

%% file: fondrugueux.tex
\begin{picture}(0,0)%
\includegraphics{fondrugueux.ps_tex}%
\end{picture}%
\setlength{\unitlength}{2072sp}%
\begingroup\makeatletter\ifx\SetFigFont\undefined%
\gdef\SetFigFont#1#2#3#4#5{%
  \reset@font\fontsize{#1}{#2pt}%
  \fontfamily{#3}\fontseries{#4}\fontshape{#5}%
  \selectfont}%
\fi\endgroup%
\begin{picture}(10377,3310)(571,-3275)
\put(5266,-376){\makebox(0,0)[lb]{\smash{{\SetFigFont{9}{10.8}{\rmdefault}{\mddefault}{\updefault}{\color[rgb]{0,0,0}$x_2$}%
}}}}
\put(6481,-1591){\makebox(0,0)[b]{\smash{{\SetFigFont{9}{10.8}{\rmdefault}{\mddefault}{\updefault}{\color[rgb]{0,0,0}$\Omega^0$}%
}}}}
\put(6481,-691){\makebox(0,0)[b]{\smash{{\SetFigFont{9}{10.8}{\rmdefault}{\mddefault}{\updefault}{\color[rgb]{0,0,0}$\Gamma^1$}%
}}}}
\put(6481,-2626){\makebox(0,0)[b]{\smash{{\SetFigFont{9}{10.8}{\rmdefault}{\mddefault}{\updefault}{\color[rgb]{0,0,0}$\Gamma^0$}%
}}}}
\put(7696,-2221){\makebox(0,0)[lb]{\smash{{\SetFigFont{9}{10.8}{\rmdefault}{\mddefault}{\updefault}{\color[rgb]{0,0,0}$x_1$}%
}}}}
\put(1576,-376){\makebox(0,0)[lb]{\smash{{\SetFigFont{9}{10.8}{\rmdefault}{\mddefault}{\updefault}{\color[rgb]{0,0,0}$x_2$}%
}}}}
\put(2791,-1591){\makebox(0,0)[b]{\smash{{\SetFigFont{9}{10.8}{\rmdefault}{\mddefault}{\updefault}{\color[rgb]{0,0,0}$\Omega^\epsilon$}%
}}}}
\put(586,-2401){\makebox(0,0)[lb]{\smash{{\SetFigFont{9}{10.8}{\rmdefault}{\mddefault}{\updefault}{\color[rgb]{0,0,0}$x_2=0$}%
}}}}
\put(586,-781){\makebox(0,0)[lb]{\smash{{\SetFigFont{9}{10.8}{\rmdefault}{\mddefault}{\updefault}{\color[rgb]{0,0,0}$x_2=1$}%
}}}}
\put(9316,-2626){\makebox(0,0)[lb]{\smash{{\SetFigFont{9}{10.8}{\rmdefault}{\mddefault}{\updefault}{\color[rgb]{0,0,0}$P$}%
}}}}
\put(9046,-196){\makebox(0,0)[lb]{\smash{{\SetFigFont{9}{10.8}{\rmdefault}{\mddefault}{\updefault}{\color[rgb]{0,0,0}$y_2$}%
}}}}
\put(10801,-2671){\makebox(0,0)[lb]{\smash{{\SetFigFont{9}{10.8}{\rmdefault}{\mddefault}{\updefault}{\color[rgb]{0,0,0}$y_1$}%
}}}}
\put(9856,-2221){\makebox(0,0)[b]{\smash{{\SetFigFont{9}{10.8}{\rmdefault}{\mddefault}{\updefault}{\color[rgb]{0,0,0}$\Gamma$}%
}}}}
\put(2791,-736){\makebox(0,0)[b]{\smash{{\SetFigFont{9}{10.8}{\rmdefault}{\mddefault}{\updefault}{\color[rgb]{0,0,0}$\Gamma^1$}%
}}}}
\put(9856,-3076){\makebox(0,0)[b]{\smash{{\SetFigFont{9}{10.8}{\rmdefault}{\mddefault}{\updefault}{\color[rgb]{0,0,0}$P^0$}%
}}}}
\put(2791,-2806){\makebox(0,0)[b]{\smash{{\SetFigFont{9}{10.8}{\rmdefault}{\mddefault}{\updefault}{\color[rgb]{0,0,0}$\Gamma^\epsilon$}%
}}}}
\put(4006,-2221){\makebox(0,0)[lb]{\smash{{\SetFigFont{9}{10.8}{\rmdefault}{\mddefault}{\updefault}{\color[rgb]{0,0,0}$x_1$}%
}}}}
\put(1441,-3166){\makebox(0,0)[b]{\smash{{\SetFigFont{9}{10.8}{\rmdefault}{\mddefault}{\updefault}{\color[rgb]{0,0,0}$x_1=0$}%
}}}}
\put(4186,-3121){\makebox(0,0)[b]{\smash{{\SetFigFont{9}{10.8}{\rmdefault}{\mddefault}{\updefault}{\color[rgb]{0,0,0}$x_1=L$}%
}}}}
\put(1306,-1591){\makebox(0,0)[b]{\smash{{\SetFigFont{9}{10.8}{\rmdefault}{\mddefault}{\updefault}{\color[rgb]{0,0,0}$\Gamma_{\rm in}$}%
}}}}
\put(9856,-1591){\makebox(0,0)[b]{\smash{{\SetFigFont{9}{10.8}{\rmdefault}{\mddefault}{\updefault}{\color[rgb]{0,0,0}$Z^+$}%
}}}}
\put(8866,-1591){\makebox(0,0)[b]{\smash{{\SetFigFont{9}{10.8}{\rmdefault}{\mddefault}{\updefault}{\color[rgb]{0,0,0}$\Gamma_{l}$}%
}}}}
\put(7991,-1591){\makebox(0,0)[b]{\smash{{\SetFigFont{9}{10.8}{\rmdefault}{\mddefault}{\updefault}{\color[rgb]{0,0,0}$\Gamma_{\rm out}'$}%
}}}}
\put(4321,-1591){\makebox(0,0)[b]{\smash{{\SetFigFont{9}{10.8}{\rmdefault}{\mddefault}{\updefault}{\color[rgb]{0,0,0}$\Gamma_{\rm out}$}%
}}}}
\put(5176,-1591){\makebox(0,0)[b]{\smash{{\SetFigFont{9}{10.8}{\rmdefault}{\mddefault}{\updefault}{\color[rgb]{0,0,0}$\Gamma_{\rm in}'$}%
}}}}
\end{picture}%

%% file: fond_qp.tex
\begin{picture}(0,0)%
\includegraphics{fond_qp.ps_tex}%
\end{picture}%
\setlength{\unitlength}{2072sp}%
\begingroup\makeatletter\ifx\SetFigFont\undefined%
\gdef\SetFigFont#1#2#3#4#5{%
  \reset@font\fontsize{#1}{#2pt}%
  \fontfamily{#3}\fontseries{#4}\fontshape{#5}%
  \selectfont}%
\fi\endgroup%
\begin{picture}(11190,3621)(706,-3541)
\put(811,-151){\makebox(0,0)[lb]{\smash{{\SetFigFont{9}{10.8}{\rmdefault}{\mddefault}{\updefault}{\color[rgb]{0,0,0}$y_2$}%
}}}}
\put(721,-1456){\makebox(0,0)[lb]{\smash{{\SetFigFont{9}{10.8}{\rmdefault}{\mddefault}{\updefault}{\color[rgb]{0,0,0}E}%
}}}}
\put(3736,-3526){\makebox(0,0)[lb]{\smash{{\SetFigFont{9}{10.8}{\rmdefault}{\mddefault}{\updefault}{\color[rgb]{0,0,0}B}%
}}}}
\put(1846,-1636){\makebox(0,0)[b]{\smash{{\SetFigFont{9}{10.8}{\rmdefault}{\mddefault}{\updefault}{\color[rgb]{0,0,0}$Z^+$}%
}}}}
\put(1306,-2671){\makebox(0,0)[lb]{\smash{{\SetFigFont{9}{10.8}{\rmdefault}{\mddefault}{\updefault}{\color[rgb]{0,0,0}$P$}%
}}}}
\put(1846,-2266){\makebox(0,0)[b]{\smash{{\SetFigFont{9}{10.8}{\rmdefault}{\mddefault}{\updefault}{\color[rgb]{0,0,0}$\Gamma$}%
}}}}
\put(7246,-2806){\makebox(0,0)[lb]{\smash{{\SetFigFont{9}{10.8}{\rmdefault}{\mddefault}{\updefault}{\color[rgb]{0,0,0}$y_1$}%
}}}}
\put(4996,-2671){\makebox(0,0)[b]{\smash{{\SetFigFont{9}{10.8}{\rmdefault}{\mddefault}{\updefault}{\color[rgb]{0,0,0}$\dots$}%
}}}}
\put(6571,-1591){\makebox(0,0)[b]{\smash{{\SetFigFont{9}{10.8}{\rmdefault}{\mddefault}{\updefault}{\color[rgb]{0,0,0}$Z^++ k\eu$}%
}}}}
\put(6571,-2221){\makebox(0,0)[b]{\smash{{\SetFigFont{9}{10.8}{\rmdefault}{\mddefault}{\updefault}{\color[rgb]{0,0,0}$\Gamma$}%
}}}}
\put(3781,-556){\makebox(0,0)[lb]{\smash{{\SetFigFont{9}{10.8}{\rmdefault}{\mddefault}{\updefault}{\color[rgb]{0,0,0}$\Pi$}%
}}}}
\put(3421,-1636){\makebox(0,0)[b]{\smash{{\SetFigFont{9}{10.8}{\rmdefault}{\mddefault}{\updefault}{\color[rgb]{0,0,0}$Z^++\eu$}%
}}}}
\put(3421,-2671){\makebox(0,0)[b]{\smash{{\SetFigFont{9}{10.8}{\rmdefault}{\mddefault}{\updefault}{\color[rgb]{0,0,0}$P+\eu$}%
}}}}
\put(6571,-2626){\makebox(0,0)[b]{\smash{{\SetFigFont{9}{10.8}{\rmdefault}{\mddefault}{\updefault}{\color[rgb]{0,0,0}$P+ k\eu$}%
}}}}
\put(8821,-196){\makebox(0,0)[lb]{\smash{{\SetFigFont{9}{10.8}{\rmdefault}{\mddefault}{\updefault}{\color[rgb]{0,0,0}$y_2$}%
}}}}
\put(11881,-2491){\makebox(0,0)[lb]{\smash{{\SetFigFont{9}{10.8}{\rmdefault}{\mddefault}{\updefault}{\color[rgb]{0,0,0}$y_1$}%
}}}}
\put(10306,-1411){\makebox(0,0)[lb]{\smash{{\SetFigFont{9}{10.8}{\rmdefault}{\mddefault}{\updefault}{\color[rgb]{0,0,0}$\Pi'$}%
}}}}
\put(8776,-1366){\makebox(0,0)[lb]{\smash{{\SetFigFont{9}{10.8}{\rmdefault}{\mddefault}{\updefault}{\color[rgb]{0,0,0}E'}%
}}}}
\put(10261,-2716){\makebox(0,0)[lb]{\smash{{\SetFigFont{9}{10.8}{\rmdefault}{\mddefault}{\updefault}{\color[rgb]{0,0,0}B'}%
}}}}
\put(9091,-2266){\makebox(0,0)[b]{\smash{{\SetFigFont{9}{10.8}{\rmdefault}{\mddefault}{\updefault}{\color[rgb]{0,0,0}$\ed$}%
}}}}
\put(9316,-2581){\makebox(0,0)[b]{\smash{{\SetFigFont{9}{10.8}{\rmdefault}{\mddefault}{\updefault}{\color[rgb]{0,0,0}$\eu$}%
}}}}
\end{picture}%

%% file: dbdnE.tex
\begin{picture}(0,0)%
\includegraphics{dbdnE.ps_tex}%
\end{picture}%
\setlength{\unitlength}{1973sp}%
\begingroup\makeatletter\ifx\SetFigFont\undefined%
\gdef\SetFigFont#1#2#3#4#5{%
  \reset@font\fontsize{#1}{#2pt}%
  \fontfamily{#3}\fontseries{#4}\fontshape{#5}%
  \selectfont}%
\fi\endgroup%
\begin{picture}(5764,3579)(1315,-4012)
\put(1638,-3648){\makebox(0,0)[rb]{\smash{{\SetFigFont{5}{6.0}{\familydefault}{\mddefault}{\updefault}-0.5}}}}
\put(1638,-3341){\makebox(0,0)[rb]{\smash{{\SetFigFont{5}{6.0}{\familydefault}{\mddefault}{\updefault}-0.4}}}}
\put(1638,-3033){\makebox(0,0)[rb]{\smash{{\SetFigFont{5}{6.0}{\familydefault}{\mddefault}{\updefault}-0.3}}}}
\put(1638,-2726){\makebox(0,0)[rb]{\smash{{\SetFigFont{5}{6.0}{\familydefault}{\mddefault}{\updefault}-0.2}}}}
\put(1638,-2418){\makebox(0,0)[rb]{\smash{{\SetFigFont{5}{6.0}{\familydefault}{\mddefault}{\updefault}-0.1}}}}
\put(1638,-2111){\makebox(0,0)[rb]{\smash{{\SetFigFont{5}{6.0}{\familydefault}{\mddefault}{\updefault} 0}}}}
\put(1638,-1803){\makebox(0,0)[rb]{\smash{{\SetFigFont{5}{6.0}{\familydefault}{\mddefault}{\updefault} 0.1}}}}
\put(1638,-1496){\makebox(0,0)[rb]{\smash{{\SetFigFont{5}{6.0}{\familydefault}{\mddefault}{\updefault} 0.2}}}}
\put(1638,-1188){\makebox(0,0)[rb]{\smash{{\SetFigFont{5}{6.0}{\familydefault}{\mddefault}{\updefault} 0.3}}}}
\put(1638,-880){\makebox(0,0)[rb]{\smash{{\SetFigFont{5}{6.0}{\familydefault}{\mddefault}{\updefault} 0.4}}}}
\put(1638,-573){\makebox(0,0)[rb]{\smash{{\SetFigFont{5}{6.0}{\familydefault}{\mddefault}{\updefault} 0.5}}}}
\put(2123,-3773){\makebox(0,0)[b]{\smash{{\SetFigFont{5}{6.0}{\familydefault}{\mddefault}{\updefault}-0.6}}}}
\put(2732,-3773){\makebox(0,0)[b]{\smash{{\SetFigFont{5}{6.0}{\familydefault}{\mddefault}{\updefault}-0.4}}}}
\put(3340,-3773){\makebox(0,0)[b]{\smash{{\SetFigFont{5}{6.0}{\familydefault}{\mddefault}{\updefault}-0.2}}}}
\put(3948,-3773){\makebox(0,0)[b]{\smash{{\SetFigFont{5}{6.0}{\familydefault}{\mddefault}{\updefault} 0}}}}
\put(4556,-3773){\makebox(0,0)[b]{\smash{{\SetFigFont{5}{6.0}{\familydefault}{\mddefault}{\updefault} 0.2}}}}
\put(5164,-3773){\makebox(0,0)[b]{\smash{{\SetFigFont{5}{6.0}{\familydefault}{\mddefault}{\updefault} 0.4}}}}
\put(5772,-3773){\makebox(0,0)[b]{\smash{{\SetFigFont{5}{6.0}{\familydefault}{\mddefault}{\updefault} 0.6}}}}
\put(6380,-3773){\makebox(0,0)[b]{\smash{{\SetFigFont{5}{6.0}{\familydefault}{\mddefault}{\updefault} 0.8}}}}
\put(6988,-3773){\makebox(0,0)[b]{\smash{{\SetFigFont{5}{6.0}{\familydefault}{\mddefault}{\updefault} 1}}}}
\put(4350,-3960){\makebox(0,0)[b]{\smash{{\SetFigFont{5}{6.0}{\familydefault}{\mddefault}{\updefault}$y_2$}}}}
\put(6388,-812){\makebox(0,0)[rb]{\smash{{\SetFigFont{5}{6.0}{\familydefault}{\mddefault}{\updefault}$\ddn{\tinlh}$}}}}
\put(6388,-1140){\makebox(0,0)[rb]{\smash{{\SetFigFont{5}{6.0}{\familydefault}{\mddefault}{\updefault}$\ddn{\beta^L_h}$}}}}
\end{picture}%

%% file: mshsize.tex
\begin{picture}(0,0)%
\includegraphics{mshsize.ps_tex}%
\end{picture}%
\setlength{\unitlength}{1973sp}%
\begingroup\makeatletter\ifx\SetFigFont\undefined%
\gdef\SetFigFont#1#2#3#4#5{%
  \reset@font\fontsize{#1}{#2pt}%
  \fontfamily{#3}\fontseries{#4}\fontshape{#5}%
  \selectfont}%
\fi\endgroup%
\begin{picture}(5838,3579)(1241,-4012)
\put(1988,-3648){\makebox(0,0)[rb]{\smash{{\SetFigFont{5}{6.0}{\familydefault}{\mddefault}{\updefault} 1e-05}}}}
\put(1988,-3033){\makebox(0,0)[rb]{\smash{{\SetFigFont{5}{6.0}{\familydefault}{\mddefault}{\updefault} 0.0001}}}}
\put(1988,-2418){\makebox(0,0)[rb]{\smash{{\SetFigFont{5}{6.0}{\familydefault}{\mddefault}{\updefault} 0.001}}}}
\put(1988,-1803){\makebox(0,0)[rb]{\smash{{\SetFigFont{5}{6.0}{\familydefault}{\mddefault}{\updefault} 0.01}}}}
\put(1988,-1188){\makebox(0,0)[rb]{\smash{{\SetFigFont{5}{6.0}{\familydefault}{\mddefault}{\updefault} 0.1}}}}
\put(1988,-573){\makebox(0,0)[rb]{\smash{{\SetFigFont{5}{6.0}{\familydefault}{\mddefault}{\updefault} 1}}}}
\put(2063,-3773){\makebox(0,0)[b]{\smash{{\SetFigFont{5}{6.0}{\familydefault}{\mddefault}{\updefault} 0.1}}}}
\put(6988,-3773){\makebox(0,0)[b]{\smash{{\SetFigFont{5}{6.0}{\familydefault}{\mddefault}{\updefault} 1}}}}
\put(1356,-2049){\rotatebox{90.0}{\makebox(0,0)[b]{\smash{{\SetFigFont{5}{6.0}{\familydefault}{\mddefault}{\updefault}$h$}}}}}
\put(4525,-3960){\makebox(0,0)[b]{\smash{{\SetFigFont{5}{6.0}{\familydefault}{\mddefault}{\updefault}$\epsilon$}}}}
\put(6388,-3081){\makebox(0,0)[rb]{\smash{{\SetFigFont{5}{6.0}{\familydefault}{\mddefault}{\updefault}$h_{\min}$}}}}
\put(6388,-3409){\makebox(0,0)[rb]{\smash{{\SetFigFont{5}{6.0}{\familydefault}{\mddefault}{\updefault}$h_{\max}$}}}}
\end{picture}%

%% file: errl2_dd_coarse.tex
\begin{picture}(0,0)%
\includegraphics{errl2_dd_coarse.ps_tex}%
\end{picture}%
\setlength{\unitlength}{1973sp}%
\begingroup\makeatletter\ifx\SetFigFont\undefined%
\gdef\SetFigFont#1#2#3#4#5{%
  \reset@font\fontsize{#1}{#2pt}%
  \fontfamily{#3}\fontseries{#4}\fontshape{#5}%
  \selectfont}%
\fi\endgroup%
\begin{picture}(5838,3579)(1241,-4012)
\put(1988,-3648){\makebox(0,0)[rb]{\smash{{\SetFigFont{5}{6.0}{\familydefault}{\mddefault}{\updefault} 0.0001}}}}
\put(1988,-2879){\makebox(0,0)[rb]{\smash{{\SetFigFont{5}{6.0}{\familydefault}{\mddefault}{\updefault} 0.001}}}}
\put(1988,-2110){\makebox(0,0)[rb]{\smash{{\SetFigFont{5}{6.0}{\familydefault}{\mddefault}{\updefault} 0.01}}}}
\put(1988,-1342){\makebox(0,0)[rb]{\smash{{\SetFigFont{5}{6.0}{\familydefault}{\mddefault}{\updefault} 0.1}}}}
\put(1988,-573){\makebox(0,0)[rb]{\smash{{\SetFigFont{5}{6.0}{\familydefault}{\mddefault}{\updefault} 1}}}}
\put(2063,-3773){\makebox(0,0)[b]{\smash{{\SetFigFont{5}{6.0}{\familydefault}{\mddefault}{\updefault} 0.1}}}}
\put(6988,-3773){\makebox(0,0)[b]{\smash{{\SetFigFont{5}{6.0}{\familydefault}{\mddefault}{\updefault} 1}}}}
\put(1356,-2049){\rotatebox{90.0}{\makebox(0,0)[b]{\smash{{\SetFigFont{5}{6.0}{\familydefault}{\mddefault}{\updefault}$\nrm{\cdot}{L^2(\Oz)}$}}}}}
\put(4525,-3960){\makebox(0,0)[b]{\smash{{\SetFigFont{5}{6.0}{\familydefault}{\mddefault}{\updefault}$\epsilon$}}}}
\end{picture}%

%% file: errh1_dd_coarse.tex
\begin{picture}(0,0)%
\includegraphics{errh1_dd_coarse.ps_tex}%
\end{picture}%
\setlength{\unitlength}{1973sp}%
\begingroup\makeatletter\ifx\SetFigFont\undefined%
\gdef\SetFigFont#1#2#3#4#5{%
  \reset@font\fontsize{#1}{#2pt}%
  \fontfamily{#3}\fontseries{#4}\fontshape{#5}%
  \selectfont}%
\fi\endgroup%
\begin{picture}(5838,3579)(1241,-4012)
\put(1988,-3648){\makebox(0,0)[rb]{\smash{{\SetFigFont{5}{6.0}{\familydefault}{\mddefault}{\updefault} 0.0001}}}}
\put(1988,-2879){\makebox(0,0)[rb]{\smash{{\SetFigFont{5}{6.0}{\familydefault}{\mddefault}{\updefault} 0.001}}}}
\put(1988,-2110){\makebox(0,0)[rb]{\smash{{\SetFigFont{5}{6.0}{\familydefault}{\mddefault}{\updefault} 0.01}}}}
\put(1988,-1342){\makebox(0,0)[rb]{\smash{{\SetFigFont{5}{6.0}{\familydefault}{\mddefault}{\updefault} 0.1}}}}
\put(1988,-573){\makebox(0,0)[rb]{\smash{{\SetFigFont{5}{6.0}{\familydefault}{\mddefault}{\updefault} 1}}}}
\put(2063,-3773){\makebox(0,0)[b]{\smash{{\SetFigFont{5}{6.0}{\familydefault}{\mddefault}{\updefault} 0.1}}}}
\put(6988,-3773){\makebox(0,0)[b]{\smash{{\SetFigFont{5}{6.0}{\familydefault}{\mddefault}{\updefault} 1}}}}
\put(1356,-2049){\rotatebox{90.0}{\makebox(0,0)[b]{\smash{{\SetFigFont{5}{6.0}{\familydefault}{\mddefault}{\updefault}$\nrm{\cdot}{H^1(\Oz)}$}}}}}
\put(4525,-3960){\makebox(0,0)[b]{\smash{{\SetFigFont{5}{6.0}{\familydefault}{\mddefault}{\updefault}$\epsilon$}}}}
\put(6388,-2425){\makebox(0,0)[rb]{\smash{{\SetFigFont{5}{6.0}{\familydefault}{\mddefault}{\updefault}$u^\epsilon_h-\uiuep$}}}}
\put(6388,-2753){\makebox(0,0)[rb]{\smash{{\SetFigFont{5}{6.0}{\familydefault}{\mddefault}{\updefault}$u^\epsilon_h-\uiue$}}}}
\put(6388,-3081){\makebox(0,0)[rb]{\smash{{\SetFigFont{5}{6.0}{\familydefault}{\mddefault}{\updefault}$u^\epsilon_h-u^1$}}}}
\put(6388,-3409){\makebox(0,0)[rb]{\smash{{\SetFigFont{5}{6.0}{\familydefault}{\mddefault}{\updefault}$u^\epsilon_h-u^0$}}}}
\end{picture}%

%% file: main.bbl
\begin{thebibliography}{10}

\bibitem{AmGiGiI.94}
C.~Amrouche, V.~Girault, and J.~Giroire.
\newblock Weighted sobolev spaces and laplace's equation in {$\R^n$}.
\newblock {\em Journal des Mathematiques Pures et Appliquees}, 73:579--606,
  January 1994.

\bibitem{AmGiGiII.97}
C.~Amrouche, V.~Girault, and J.~Giroire.
\newblock Dirichlet and neumann exterior problems for the n-dimensional laplace
  operator an approach in weighted sobolev spaces.
\newblock {\em Journal des Mathematiques Pures et Appliquees}, 76:55--81(27),
  January 1997.

\bibitem{BrBoMi}
E.~Bonnetier, D.~Bresch, and V.~Milisic.
\newblock Blood flow modelling in stented arteries: new convergence results of
  first order boundary layers and wall-laws for a rough neumann-laplace
  problem.
\newblock submitted.

\bibitem{BrMiQam}
D.~Bresch and V.~Milisic.
\newblock High order multi-scale wall laws~: part i, the periodic case.
\newblock accepted for publication in Quart. Appl. Math. 2008.

\bibitem{BrMiCras}
D.~Bresch and V.~Milisic.
\newblock Towards implicit multi-scale wall laws.
\newblock accepted for publication in C. R. Acad. Sciences, S{\'e}rie
  Math{\'e}matiques, 2008.

\bibitem{Evans.Book}
L.~C. Evans.
\newblock {\em Partial differential equations}, volume~19 of {\em Graduate
  Studies in Mathematics}.
\newblock American Mathematical Society, Providence, RI, 1998.

\bibitem{GiTru.Book}
D.~Gilbarg and N.~S. Trudinger.
\newblock {\em Elliptic partial differential equations of second order}.
\newblock Classics in Mathematics. Springer-Verlag, Berlin, 2001.
\newblock Reprint of the 1998 edition.

\bibitem{Grisvard}
P.~Grisvard.
\newblock {\em Elliptic problems in nonsmooth domains}, volume~24 of {\em
  Monographs and Studies in Mathematics}.
\newblock Pitman (Advanced Publishing Program), Boston, MA, 1985.

\bibitem{Ha.71}
B.~Hanouzet.
\newblock Espaces de {S}obolev avec poids application au probl{\`e}me de
  {D}irichlet dans un demi espace.
\newblock {\em Rend. Sem. Mat. Univ. Padova}, 46:227--272, 1971.

\bibitem{ffm}
F.~Hecht, O.~Pironneau, A.~Le~Hyaric, and Ohtsuka K.
\newblock {\em Freefem++}.
\newblock Laboratoire Jacques-Louis Lions, Universite Pierre et Marie Curie,
  Paris, 2005.

\bibitem{JaMiSIAM.00}
W.~J{\"a}ger and A.~Mikeli{\'c}.
\newblock {On the interface boundary condition of Beavers, Joseph, and
  Saffman.}
\newblock {\em SIAM J. Appl. Math.}, 60(4):1111--1127, 2000.

\bibitem{JaMiJDE.01}
W.~J{\"a}ger and A.~Mikeli{\'c}.
\newblock On the roughness-induced effective boundary condition for an
  incompressible viscous flow.
\newblock {\em J. Diff. Equa.}, 170:96--122, 2001.

\bibitem{JaMiNe.01}
W.~J{\"a}ger, A.~Mikeli{\'c}, and N.~Neuss.
\newblock {Asymptotic analysis of the laminar viscous flow over a porous bed.}
\newblock {\em SIAM J. Sci. Comput.}, 22(6):2006--2028, 2001.

\bibitem{JaMiPise.96}
Willi J{\"a}ger and Andro Mikeli{\'c}.
\newblock On the boundary conditions at the contact interface between a porous
  medium and a free fluid.
\newblock {\em Ann. Scuola Norm. Sup. Pisa Cl. Sci. (4)}, 23(3):403--465, 1996.

\bibitem{Ku.80.book}
A.~Kufner.
\newblock {\em Weighted Sobolev spaces.}
\newblock BSB B. G. Teubner Verlagsgesellschaft, teubner-texte zur mathematik
  edition, 1980.

\bibitem{Ne.Book.67}
J.~Ne{\v{c}}as.
\newblock {\em Les m{\'e}thodes directes en th{\'e}orie des {\'e}quations
  elliptiques}.
\newblock Masson et Cie, {\'E}diteurs, Paris, 1967.

\bibitem{NeNeMi.06}
N.~Neuss, M.~Neuss-Radu, and A.~Mikeli{\'c}.
\newblock Effective laws for the poisson equation on domains with curved
  oscillating boundaries.
\newblock {\em Applicable Analysis}, 85:479--502, 2006.

\bibitem{QuVaBook.99}
A.~Quarteroni and A.~Valli.
\newblock {\em {Domain decomposition methods for partial differential
  equations.}}
\newblock Numerical Mathematics and Scientific Computation. Oxford Science
  Publication, Oxford, 1999.

\bibitem{SaPeZa.85}
E.~Sanchez-Palencia and A.~Zaoui.
\newblock {\em Homogenization techniques for composite media}, volume 272 of
  {\em Lecture Notes in Physics}.
\newblock Springer-Verlag., 1987.

\end{thebibliography}
